\documentclass[10pt]{article}
\usepackage{epsf}
\usepackage{amsmath}

\allowdisplaybreaks

\usepackage[showframe=false]{geometry}
\usepackage{changepage}

\usepackage{epsfig}
\usepackage{amssymb}

\usepackage{amsthm}
\usepackage{setspace}
\usepackage{cite}

\usepackage{mcite}

\usepackage{algorithmic}  
\usepackage{algorithm}

\usepackage{shadow}
\usepackage{fancybox}
\usepackage{fancyhdr}

\usepackage{color}
\usepackage[usenames,dvipsnames,svgnames,table]{xcolor}
\newcommand{\bl}[1]{\textcolor{blue}{#1}}

\definecolor{mypurple}{rgb}{.4,.0,.5}

\usepackage[hyphens]{url}

\usepackage[colorlinks=true,
            linkcolor=black,
            urlcolor=blue,
            citecolor=purple]{hyperref}

\usepackage{breakurl}

\def\w{{\bf w}}

\def\y{{\bf y}}

\def\x{{\bf x}}

\def\x{{\mathbf x}}

\def\w{{\bf w}}

\def\x{{\bf x}}
\def\y{{\bf y}}

\def\be{\begin{equation}}
\def\ee{\end{equation}}
\def\ba{\left[\begin{array}}
\def\ea{\end{array}\right]}

\def\w{{\bf w}}

\def\x{{\bf x}}
\def\y{{\bf y}}

\def\1{{\bf 1}}

\def\g{{\bf g}}
\def\0{{\bf 0}}

\def\erf{\mbox{erf}}
\def\erfc{\mbox{erfc}}
\def\erfi{\mbox{erfi}}

\def\kmu{k_{\mu}}

\def\mR{{\mathbb R}}
\def\mN{{\mathbb N}}
\def\mE{{\mathbb E}}

\def\psiint{\Psi_{int}}
\def\psiext{\Psi_{ext}}
\def\psicom{\Psi_{com}}

\def\phiint{\phi_{int}}
\def\phiext{\phi_{ext}}

\def\calF{{\cal F}}

\def\lp{\left (}
\def\rp{\right )}

\newtheorem{theorem}{Theorem}

\begin{document}

\begin{singlespace}

\title {Box constrained $\ell_1$ optimization in random linear systems -- finite dimensions 
}
\author{
\textsc{Mihailo Stojnic
\footnote{e-mail: {\tt flatoyer@gmail.com}} }}
\date{}
\maketitle

\centerline{{\bf Abstract}} \vspace*{0.1in}

Our companion work \cite{Stojnicl1BnBxasymldp} considers random under-determined linear systems with box-constrained sparse solutions and provides an asymptotic analysis of a couple of modified $\ell_1$ heuristics adjusted to handle such systems (we refer to these modifications of the standard $\ell_1$ as binary and box $\ell_1$). Our earlier work \cite{StojnicISIT2010binary} established that the binary $\ell_1$ does exhibit the so-called phase-transition phenomenon (basically the same phenomenon well-known through earlier considerations to be a key feature of the standard $\ell_1$, see, e.g. \cite{DonohoPol,DonohoUnsigned,StojnicCSetam09,StojnicUpper10}). Moreover, in \cite{StojnicISIT2010binary}, we determined the precise location of the co-called phase-transition (PT) curve. On the other hand, in \cite{Stojnicl1BnBxasymldp} we provide a much deeper understanding of the PTs and do so through a large deviations principles (LDP) type of analysis. In this paper we complement the results of \cite{Stojnicl1BnBxasymldp} by leaving the asymptotic regime naturally assumed in the PT and LDP considerations aside and instead working in a finite dimensional setting. Along the same lines, we provide for both, the binary and the box $\ell_1$, precise finite dimensional analyses and essentially determine their ultimate statistical performance characterizations. On top of that, we explain how the results created here can be utilized in the asymptotic setting, considered in \cite{Stojnicl1BnBxasymldp}, as well. Finally, for the completeness, we also present a collection of results obtained through numerical simulations and observe that they are in a massive agreement with our theoretical calculations.

\vspace*{0.25in} \noindent {\bf Index Terms: Finite dimensions; box-constrained $\ell_1$; linear systems; sparse solutions}.

\end{singlespace}

\section{Introduction}
\label{sec:back}

Linear systems of equations are of course a research arena that has attracted a substantial amount of interests from various scientific fields over last a couple of centuries. Many of their features over time became very prominent research topics and quite a few of them eventually reached the level to even be considered separate scientific fields in their own right. One of them that has been particulary intriguing over last several decades relates to a special type of these systems, namely the under-determined ones with sparse solutions. In this paper we will study precisely such systems while additionally assuming that they have binary/box-constrained solutions. Even in such a particular setup a lot of research has been done and as the problems that we will study here will be pretty much identical to those that we studied in \cite{Stojnicl1BnBxasymldp,StojnicISIT2010binary} we will not go into too much detailing about their various descriptions, history, and importance (instead, for more on these features, we will often refer to \cite{Stojnicl1BnBxasymldp,StojnicISIT2010binary} and references therein). To avoid unnecessary repetitions we below focus just on as minimal problems' reintroductions as we deem needed.

As usual, we assume that there is an $m\times n$ ($m\leq n$) dimensional real system matrix $A$. Also, let $\tilde{\x}$ be an $n$ dimensional real vector (for short we say, $A\in \mR^{m\times n}$ and $\tilde{\x}\in \mR^{n}$). A vector will be called $k$-sparse if no more than $k$ of its entries are different from zero. Set
\begin{equation}
\y\triangleq A\tilde{\x}. \label{eq:defy}
\end{equation}
Given $A$ and $\y$, the standard linear system inverse problem asks to determine $\tilde{\x}$. In other words one would like to solve the following linear system over unknown $\x$
\begin{equation}
A\x=\y. \label{eq:system}
\end{equation}
In the systems that we will consider here two particular types of vectors $\tilde{\x}$ will be of interest: 1) binary $k$-sparse and 2) box-constrained $k$-sparse vectors. Under the binary $k$-sparse we will consider $k$-sparse vectors that have all nonzero components equal to one. On the other hand, under the box-constrained $k$-sparse we will consider vectors that have no more than $k$ components from the interval $(0,1)$ and all other components equal to either zero or to one (there is really nothing too specific about zeros and ones and instead of them one can, for the both of the above definitions, use any two real numbers and everything that we present below can be repeated with rather minimal adjustments). For the binary case, the above linear system problem is often rewritten in the following optimization problem form
\begin{eqnarray}
\mbox{min} & & \|\x\|_{0}\nonumber \\
\mbox{subject to} & & A\x=\y\nonumber \\
&& \x\in \{0,1\}^n. \label{eq:binl0}
\end{eqnarray}
Analogously, for the box-constrained case we have
\begin{eqnarray}
\mbox{min} & & \|\min\{\x,\1-\x\}\|_{0}\nonumber \\
\mbox{subject to} & & A\x=\y\nonumber \\
&& \x\in[0,1]^n, \label{eq:boxl0}
\end{eqnarray}
where we view $\|\x\|_{0}$ as a mathematical object that counts the number of the nonzero entries of $\x$ and $\min$ is applied component-wise. The above problems are of course not super easy and designing efficient (basically fast and if possible polynomial) algorithms that can handle them is of particular interest. In fact, a little bit more is what one typically wants in these types of linear systems problems. Namely, assuming that $\tilde{\x}$ in (\ref{eq:defy}) is $k$-sparse ($k\leq n$) one does not necessarily need the number of equations to be equal to $n$ (from this point on we always assume that $A$ is full rank, either deterministically or when random statistically and that the systems' descriptions are such that the solutions that one is looking for are unique). A smaller number of equations is often sufficient (in fact, in statistical scenarios that we will consider here, the number of equations $m$ can often be as small as $k$). Along these lines, what one is typically looking for when facing problems (\ref{eq:binl0}) and (\ref{eq:boxl0}) is not only a fast/polynomial algorithm but also an algorithm that can recover vector that is $k$-sparse with as few equations as possible. One then can be a bit more greedy and insist that the algorithms \emph{provably} recover $k$-sparse vectors assuming that the system has a certain number of equations. In our view, the best algorithms that fit this description, come from the class of the so-called $\ell_1$-relaxations. For problems (\ref{eq:binl0}) and (\ref{eq:boxl0}) such a relaxation is given by the following (see also, e.g. \cite{Stojnicl1BnBxasymldp,StojnicISIT2010binary} and references therein)
\begin{eqnarray}
\mbox{min} & & \|\x\|_{1}\nonumber \\
\mbox{subject to} & & A\x=\y\nonumber \\
&& \x\in[0,1]^n. \label{eq:l1}
\end{eqnarray}
These types of relaxations have gained a lot popularity in last several decades. In our view that is in particular due to the following two of their features: 1) they are relatively fast and easy to implement as they represent nothing more than simple linear programs and 2) one can characterize in a mathematically rigorous fashion their performance (we should also add that there are quite a few fairly successful algorithms developed over last several decades as alternatives to the $\ell_1$ relaxations, see, e.g. \cite{JATGomp,NeVe07,DTDSomp,NT08,DaiMil08,DonMalMon09}; as mentioned above, our favorites though are precisely the $\ell_1$ relaxations, one of which, (\ref{eq:l1}), we will study here as well).

Before we switch to the presentation of our main results, we will briefly recall on what we believe are mathematically the most important/relevant aspects of the above mentioned performance characterizations that are already known. We first recall that a first substantial progress in recent years in theoretically characterizing $\ell_1$ relaxations appeared in \cite{CRT,DOnoho06CS} where in a statistical scenario it was shown that the standard $\ell_1$ (basically the one from (\ref{eq:l1}) without the interval constraint) can recover a linear sparsity (recovering linear sparsity simply means that if the system's dimensions are large and $m$ is proportional to $n$ then there is a $k$ also proportional to $n$ such that the $\ell_1$ can recover it). On the other hand, \cite{DonohoPol,DonohoUnsigned,StojnicCSetam09,StojnicUpper10} went much further and replaced a qualitative linearity characterization with the exact characterizations of the so-called $\ell_1$'s phase transitions (PT). Finally,  \cite{Stojnicl1RegPosasymldp} provided a complete characterization of a much stronger $\ell_1$'s large deviations principle (LDP). In \cite{StojnicISIT2010binary} it was observed that the concepts developed in \cite{StojnicCSetam09,StojnicUpper10}, can be used to analyze not only the standard $\ell_1$ but also the one that is given in (\ref{eq:l1}) (which we ill refer to as the binary $\ell_1$ when using it as a relaxation to solve (\ref{eq:binl0}) and as the box $\ell_1$ when using is as a relaxation to solve (\ref{eq:boxl0})). Moreover, \cite{StojnicISIT2010binary} determined the exact binary $\ell_1$ phase transition curve (PT curve). In our companion paper \cite{Stojnicl1BnBxasymldp}, we go much further and provide the exact binary $\ell_1$'s LDP characterization as well. Moreover, in \cite{Stojnicl1BnBxasymldp}, we also look at the box $\ell_1$ and settle its LDP characterization as well (as a simple consequence of the provided LDP analysis we then in \cite{Stojnicl1BnBxasymldp} determine the box $\ell_1$'s PT curve as well). Since both, the PT and the LDP concepts, are naturally related to large dimensional systems, our presentation in \cite{Stojnicl1BnBxasymldp} puts an emphasis on asymptotic (infinite dimensional) regimes. In this paper we go in a different direction and complement such a view by considering the finite dimensional settings. We provide exact finite dimensional performance characterizations for both, the binary and the box $\ell_1$. For both of these characterizations we also show how they can be transformed to bridge towards the asymptotic regime considered in \cite{Stojnicl1BnBxasymldp}.

The paper's presentation will be split into two main sections. The first one will deal with the binary $\ell_1$ and the second will deal with the box $\ell_1$.


\section{Binary $\ell_1$}
\label{sec:posl1}

In this section we discuss the binary $\ell_1$. Without loss of generality we will assume that the vector $\x$ that is the solution of (\ref{eq:binl0}) is such that its elements $\x_{k+1},\x_{k+2},\dots,\x_{n}$ are equal to zero and its elements $\x_{1},\x_{2},\dots,\x_k$ are equal to one. The following result is established in \cite{StojnicISIT2010binary} and is basically a binary adaptation of a result proven for the general $\ell_1$ in \cite{StojnicCSetam09,StojnicICASSP09}. As it will be clear later on, precisely this result is at the heart of everything that we were able to prove.
\begin{theorem}(\cite{StojnicISIT2010binary} Nonzero elements of $\x$ have fixed locations and are equal to one)
Assume that an $m\times n$ system matrix $A$ is given. Let $\x$
be a binary $k$-sparse vector. Also let $\x_{k+1}=\x_{k+2}=\dots=\x_{n}=0$. Clearly, $\x_{1}=\x_{2}=\dots=\x_k=1$. Further, assume that $\y=A\x$ and that $\w$ is
a $n\times 1$ vector such that $\w_i\leq 0, 1\leq i\leq k,$ and $\w_i\geq 0, k+1\leq i\leq n$. If
\begin{equation}
(\forall \w\in \textbf{R}^n | A\w=0) \quad  -\sum_{i=1}^{k} \w_i<\sum_{i=k+1}^{n}\w_{i},
\label{eq:posthmcond1}
\end{equation}
then the solutions of (\ref{eq:binl0}) and (\ref{eq:l1}) coincide. Moreover, if
\begin{equation}
(\exists \w\in \textbf{R}^n | A\w=0) \quad  -\sum_{i=1}^{k} \w_i\geq \sum_{i=k+1}^{n}\w_{i},
\label{eq:posthmcond2}
\end{equation}
then the binary $k$-sparse $\x$ is the solution of (\ref{eq:binl0}) and is not the solution of (\ref{eq:l1}).
\label{thm:posthmregposcond}
\end{theorem}
Below we will try to follow the strategy presented in \cite{Stojnicl1RegPosfinn}. To that end we will also try to avoid repeating arguments for many of the concepts introduced in  \cite{Stojnicl1RegPosfinn} that directly apply to the binary $\ell_1$ considered here. Instead we will emphasize those that are different. To facilitate the exposition we set
\begin{equation}
C^{(bin)}_w\triangleq\{\w\in \mR^n| \quad -\sum_{i=1}^k \w_i\geq \sum_{i=k+1}^{n}\w_{i},\w_i\geq 0,k+1\leq i\leq n,\w_i\leq 0,1\leq i\leq k\}.\label{eq:posdefSw}
\end{equation}
It is obvious that $C^{(bin)}_w$ is a polyhedral cone. Assuming that $A$ is random and that it has say i.i.d. standard normal components (alternatively one can consider $A$ that has the null-space uniformly distributed in the corresponding Grassmanian) and following \cite{Stojnicl1RegPosfinn} we have
\begin{equation}
p^{(bin)}_{err}(k,m,n)=2\sum_{l=m+2j+1,j\in \mN_0}^{n} \sum_{F^{(l,bin)}\in \calF^{(l,bin)}}\phiint(0,F^{(l,bin)})\phiext(F^{(l,bin)},C^{(bin)}_w),\label{eq:posanal3}
\end{equation}
where $p^{(bin)}_{err}(k,m,n)$ is the probability that the solution of (\ref{eq:l1}) is not the binary $k$-sparse solution of (\ref{eq:binl0}), $F^{(l,bin)}$ is an $l$-face of $C^{(bin)}_w$, $\calF^{(l,bin)}$ is the set of all $l$-faces of $C^{(bin)}_w$, and $\phiint(\cdot,\cdot)$ and $\phiext(\cdot,\cdot)$ are the so-called internal and external angles, respectively (as mentioned earlier, we will often assume a substantial level of familiarity with what is presented in \cite{Stojnicl1RegPosfinn}; along the same lines more on the faces and angles of polyhedral cones of similar type can be found in \cite{Stojnicl1RegPosfinn} and various references therein). As emphasized on a multitude of occasions in \cite{Stojnicl1RegPosfinn},
(\ref{eq:posanal3}) is a nice conceptual characterization of $p^{(bin)}_{err}(k,m,n)$. However, to be able to use it one would need the angles $\phiint(\cdot,\cdot)$ and $\phiext(\cdot,\cdot)$ as well. Below we compute these angles. We split the presentation into two parts, the first one that relates to the internal angles and the second one that deals with the external angles.

\subsection{Internal angles}
\label{sec:posintang}

The internal angles,  $\phiint(0,F^{(l,bin)})$, will be the subject of this section. First we note that there are several different types of $l$-faces of $C^{(bin)}_w$ and we split the set of all $l$-faces $\calF^{(l,bin)}$ into two sets, $\calF^{(l,bin)}_1$ and $\calF^{(l,bin)}_2$ in the following way. Set
\begin{eqnarray}
I_l & \triangleq  & \{1,2,\dots,k\} \nonumber \\
I_r & \triangleq  & \{k+1,k+2,\dots,n\} \nonumber \\
l_1^{(min)} & \triangleq & \max(0,n-l-1-(n-k)) \nonumber \\
l_1^{(max)} & \triangleq & \min(k-1,n-l-1) \nonumber \\
l_2^{(min)} & \triangleq & \max(0,n-l-(n-k)) \nonumber \\
l_2^{(max)} & \triangleq & \min(k-1,n-l),\label{eq:posintanal1}
\end{eqnarray}
and write
\begin{equation}
\calF^{(l,bin)}_1 \triangleq \cup_{l_1=l_1^{(min)}}^{l_1^{(max)}}\calF^{(l,l_1,bin)}_1, l\in \{0,1,\dots,n-1\},\label{eq:posintanal1a}
\end{equation}
where
\begin{multline}
\calF^{(l,l_1,bin)}_1 \triangleq\{\w\in \mR^n| \quad -\sum_{i=1}^k \w_i= \sum_{i=k+1}^{n}\w_{i},\w_{I_l}\leq 0,\w_{I_r}\geq 0, \\
\w_{I^{(l,l_1,bin)}_{1,r}}=0,I^{(l,l_1,bin)}_{1,r}\subset I_r,|I^{(l,l_1,bin)}_{1,r}|=n-l-1-l_1 \quad \mbox{and} \quad \w_{I^{(l,l_1,bin)}_{1,l}}=0,I^{(l,l_1,bin)}_{1,l}\subset I_l,|I^{(l,l_1,bin)}_{1,l}|=l_1\},\label{eq:posintanal2}
\end{multline}
and
\begin{equation}
\calF^{(l,bin)}_2 \triangleq \cup_{l_2=l_2^{(min)}}^{l_2^{(max)}}\calF^{(l,l_2,bin)}_2, l\in \{1,2,\dots,n\},\label{eq:posintanal2a}
\end{equation}
where
\begin{multline}
\calF^{(l,l_2,bin)}_1 \triangleq\{\w\in \mR^n| \quad -\sum_{i=1}^k \w_i\geq \sum_{i=k+1}^{n}\w_{i},\w_{I_l}\leq 0,\w_{I_r}\geq 0, \\
\w_{I^{(l,l_2,bin)}_{2,r}}=0,I^{(l,l_2,bin)}_{2,r}\subset I_r,|I^{(l,l_2,bin)}_{2,r}|=n-l-l_1 \quad \mbox{and} \quad \w_{I^{(l,l_2,bin)}_{2,l}}=0,I^{(l,l_2,bin)}_{2,l}\subset I_l,|I^{(l,l_2,bin)}_{2,l}|=l_1\},\label{eq:posintanal3}
\end{multline}
where notation that includes indexing by a set is as in \cite{Stojnicl1RegPosfinn}. It is not that hard to see that the cardinalities of sets $\calF^{(l,l_1,bin)}_1$ and $\calF^{(l,l_1bin)}_2$ are given by
\begin{equation}
c^{(l,l_1,bin)}_1\triangleq |\calF^{(l,l_1,bin)}_1|=\binom{k}{l_1}\binom{n-k}{n-l-1-l_1},\label{eq:posintanal4}
\end{equation}
and
\begin{equation}
c^{(l,l_2,bin)}_2\triangleq |\calF^{(l,l_2,bin)}_2|=\binom{k}{l_2}\binom{n-k}{n-l-l_2}.\label{eq:posintanal5}
\end{equation}
Using all of the above, (\ref{eq:posanal3}) can be transformed to
\begin{eqnarray}
p^{(bin)}_{err}(k,m,n) & = & 2\sum_{l=m+2j+1,j\in \mN_0}^{n} ( \sum_{F^{(l,bin)}_1\in \calF^{(l,bin)}_1}\phiint(0,F^{(l,bin)}_1)\phiext(F^{(l,bin)}_1,C^{(bin)}_w)\nonumber \\
& & + \sum_{F^{(l,bin)}_2\in \calF^{(l,bin)}_2}\phiint(0,F^{(l,bin)}_2)\phiext(F^{(l,bin)}_2,C^{(bin)}_w))\nonumber \\
& = & 2 ( \sum_{l=m+2j+1,j\in \mN_0}^{n-1} \sum_{l_1=l^{(min)}_1}^{l^{(max)}_1} c^{(l,l_1,bin)}_1\phiint(0,F^{(l,l_1,bin)}_1)\phiext(F^{(l,l_1,bin)}_1,C^{(bin)}_w) \nonumber \\
& & + \sum_{l=m+2j+1,j\in \mN_0}^{n} \sum_{l_2=l^{(min)}_2}^{l^{(max)}_2} c^{(l,l_2,bin)}_2 \phiint(0,F^{(l,l_2,bin)}_2)\phiext(F^{(l,l_2,bin)}_2,C^{(bin)}_w)),\nonumber \\
\label{eq:posintanal6}
\end{eqnarray}
Now, one can apply the same line of thinking as in \cite{Stojnicl1RegPosfinn} and due to symmetry conclude that $F^{(l,l_1,bin)}_1$ and $F^{(l,l_2,bin)}_2$ are basically any of the elements from sets $\calF^{(l,l_1,bin)}_1$ and $\calF^{(l,l_2,bin)}_2$, respectively. For the concreteness we choose
\begin{eqnarray}
  I^{(l,l_1,bin)}_{1,l} & = & \{1,2,\dots,l_1\} \nonumber\\
  I^{(l,l_1,bin)}_{1,r} & = & \{l+l_1+2,l+l_1+3,\dots,n\} \nonumber\\
  I^{(l,l_2,bin)}_{2,l} & = & \{1,2,\dots,l_2\} \nonumber\\
  I^{(l,l_2,bin)}_{2,r} & = & \{l+l_2+1,l+l_2+2,\dots,n\},\label{eq:posintanal6a}
\end{eqnarray}
and consequently
\begin{equation}
F^{(l,l_1,bin)}_1 =\{\w\in \mR^n| \quad -\sum_{i=l_1+1}^{k} \w_i= \sum_{i=k+1}^{l+1+l_1}\w_{i},\w_{l_1+1:k}\leq 0,\w_{1:l_1}=0,\w_{k+1:l+l_1+1}\geq 0,\w_{l+l_1+2:n}=0\},\label{eq:posintanal7}
\end{equation}
and
\begin{equation}
F^{(l,l_2,bin)}_2 =\{\w\in \mR^n| \quad -\sum_{i=l_2+1}^{k} \w_i\geq \sum_{i=k+1}^{l+l_2}\w_{i},\w_{l_2+1:k}\leq 0,\w_{1:l_2}=0,\w_{k+1:l+l_2}\geq 0,\w_{l+l_2+1:n}=0\}.\label{eq:posintanal8}
\end{equation}
Below we separately present the computations of $\phiint(0,F^{(l,l_1,bin)}_1)$ and $\phiint(0,F^{(l,l_2,bin)}_2)$.

\subsubsection{Computing $\phiint(0,F^{(l,l_1,bin)}_1)$}
\label{sec:posint1ang}

We will rely on the ``Gaussian coordinates in an orthonormal basis" (GCOB) strategy that we introduced in \cite{Stojnicl1RegPosfinn} to
compute $\phiint(0,F^{(l,l_1,bin)}_1)$. As usual we will skip all the repetitive details from \cite{Stojnicl1RegPosfinn} and showcase only the key differences. For the orthonormal basis we will use the column vectors of the following matrix
\begin{equation}
B^{(l,l_1,bin)}_{int,1}=\begin{bmatrix}
            \0_{l_1\times (l-1)} & \0_{l_1\times 1} \\
            \begin{bmatrix}
              B \\
              \0_{1\times (l-1)}
            \end{bmatrix}  & \begin{bmatrix}
              -\1_{l\times 1} \\
              l
            \end{bmatrix}\frac{1}{\sqrt{l^2+l}}\\
            \0_{(n-l-1-l_1)\times (l-1)} & \0_{(n-l-1-l_1)\times 1}
          \end{bmatrix}.\label{eq:posint1anal1}
\end{equation}
where $B$ is an $l\times (l-1)$ orthonormal matrix such that $\1_{1\times l} B=\0_{(l-1)\times 1}$, and $\1$ and $\0$ are matrices of all ones or zeros, respectively of the sizes given in their indexes. It is rather clear that $(B^{(l,l_1,bin)}_{int,1})^TB^{(l,l_1,bin)}_{int,1}=I$ and that  $F^{(l,l_1,bin)}_1$ is in the subspace spanned by the columns of $B^{(l,l_1,bin)}_{int,1}$. Now, while matrix $B^{(l,l_1,bin)}_{int,1}$ is different from the corresponding one in \cite{Stojnicl1RegPosfinn} we managed to design it so that it is for our purposes here basically the same as the corresponding one in \cite{Stojnicl1RegPosfinn}. Namely, carefully looking at $B^{(l,l_1,bin)}_{int,1}$ one can note that if its first $l_1$ rows are moved to become the last $l_1$ rows then $B^{(l,l_1,bin)}_{int,1}$ becomes exactly the same as the matrix that we dealt with in \cite{Stojnicl1RegPosfinn}. Such an observation is critical and enables us to immediately borrow many relevant conclusions that we achieved in \cite{Stojnicl1RegPosfinn}. Below we quickly sketch how that can be done.

As explained in \cite{Stojnicl1RegPosfinn}, the second part of the above mentioned ``Gaussian coordinates in an orthonormal basis" strategy asssumes that one works with the i.i.d. standard normal coordinates $\g$ to parameterize $F^{(l,l_1,bin)}_1$ in basis $B^{(l,l_1,bin)}_{int,1}$. In other words, one is looking for the set of $\g$'s (say, $G^{(l,l_1,bin)}_1$) such that $B^{(l,l_1,bin)}_{int,1}\g\in F^{(l,l_1,bin)}_1$. Basically, the following set of $\g$'s, $G^{(l,l_1,bin)}_1$, will be allowed
\begin{equation}
G^{(l,l_1,bin)}_1 =\{\g\in \mR^l| \w=B^{(l,l_1,bin)}_{int,1}\g \quad \mbox{and}\quad \w_{l_1+1:k}\leq 0,\w_{k+1:l+l_1+1}\geq 0\}.\label{eq:posint1anal2}
\end{equation}
Based on (\ref{eq:posint1anal1}), (\ref{eq:posint1anal2}) can be written in the following equivalent form as well
\begin{equation}
G^{(l,l_1,bin)}_1 =\{\g\in \mR^l| \w_{l_1+1:(l+l_1+1)}=\begin{bmatrix}
            \begin{bmatrix}
              B \\
              \0_{1\times (l-1)}
            \end{bmatrix}  & \begin{bmatrix}
              -\1_{l\times 1} \\
              l
            \end{bmatrix}\frac{1}{\sqrt{l^2+l}}
          \end{bmatrix}\g \quad \mbox{and}\quad \w_{l_1+1:k}\leq 0,\w_{k+1:l+l_1+1}\geq 0\}.\label{eq:posint1anal3}
\end{equation}
Now we can repeat step by step the remaining considerations from \cite{Stojnicl1RegPosfinn} to finally arrive at the following characterization of $\phiint(0,F^{(l,l_1,bin)}_1)$
\begin{eqnarray}
\phiint(0,F^{(l,l_1,bin)}_1) &  = & \frac{J}{(2\pi)^{\frac{l}{2}}}\int_{-\1_{1\times l}\w_{l_1+1:l+l_1}\geq 0,\w_{l_1+1:k}\leq 0,\w_{k+1:l+l_1}\geq 0}e^{-\frac{\|\w_{l_1+1:l+l_1}\|^2+
             ((-\1_{1\times l})\w_{l_1+1:l+l_1})^2}{2}}d\w_{l_1+1:l+l_1}\nonumber \\
&  = & \frac{J2^l}{2^l(2\pi)^{\frac{l}{2}}}\int_{-\1_{1\times l}\w_{l_1+1:l+l_1}\geq 0,\w_{l_1+1:k}\leq 0,\w_{k+1:l+l_1}\geq 0}e^{-\frac{\|\w_{l_1+1:l+l_1}\|^2+
             ((-\1_{1\times l})\w_{l_1+1:l+l_1})^2}{2}}d\w_{l_1+1:l+l_1}\nonumber \\
&  = & \frac{J}{2^{l}}\int_{x\geq 0}f_{x^{(bin)}}(x)e^{-\frac{x^2}{2}}dx,\label{eq:posint1anal8}
\end{eqnarray}
where $J=\sqrt{l+1}$ and $f_{x^{(bin)}}(\cdot)$ is the pdf of the random variable $x^{(bin)}=-\1_{1\times l}\w_{l_1+1:l+l_1}$, where $\w_i,l_1+1 \leq i\leq k$, are  i.i.d. negative standard half-normals and $\w_i,k+1 \leq i\leq l+l_1$, are  i.i.d. standard half-normals. To determine $f_{x^{(bin)}}(\cdot)$, we will first compute the characteristic function of $x^{(bin)}$. To that end we have
\begin{eqnarray}
E e^{itx^{(bin)}} & = & \frac{2^{l}}{(2\pi)^{\frac{l}{2}}}\int_{\w_{l_1+1:k}\leq 0,\w_{k+1:l+l_1}\geq 0}e^{-\frac{\w_{l_1+1:l+l_1}^T\w_{l_1+1:l+l_1}}{2}-it\1_{1\times l}\w_{l_1+1:l+l_1}}d\w_{l_1+1:l+l_1} \nonumber \\
& = & \lp 1+i\erfi\lp\frac{t}{\sqrt{2}}\rp\rp^{k-l_1} \lp 1-i\erfi\lp\frac{t}{\sqrt{2}}\rp\rp^{l-k+l_1} e^{-\frac{lt^2}{2}},\label{eq:posint1anal9}
\end{eqnarray}
and
\begin{equation}
f_{x^{(bin)}}(x)=\frac{1}{2\pi}\int_{-\infty}^{\infty} \lp 1+i\erfi\lp\frac{t}{\sqrt{2}}\rp\rp^{k-l_1} \lp 1-i\erfi\lp\frac{t}{\sqrt{2}}\rp\rp^{l-k+l_1} e^{-\frac{lt^2}{2}}e^{-itx}dt,\label{eq:posint1anal10}
\end{equation}
where
\begin{equation}
\erfi(y)=-i\erf(iy)=\frac{2}{\sqrt{\pi}}\int_{0}^{y}e^{\frac{z^2}{2}}dz.\label{eq:posint1anal10a}
\end{equation}
Combining (\ref{eq:posint1anal8}) and (\ref{eq:posint1anal10}) we obtain
\begin{eqnarray}
\phiint(0,F^{(l,l_1,bin)}_1) &  = &  \frac{J}{2^{l}}\int_{x\geq 0}f_{x^{(bin)}}(x)e^{-\frac{x^2}{2}}dx\nonumber \\
& = & \frac{J}{2^{l}}\int_{x\geq 0}\frac{1}{2\pi}\int_{-\infty}^{\infty}
\lp 1+i\erfi\lp\frac{t}{\sqrt{2}}\rp\rp^{k-l_1}
\lp 1-i\erfi\lp\frac{t}{\sqrt{2}}\rp\rp^{l-k+l_1} e^{-\frac{lt^2}{2}}e^{-itx}dt e^{-\frac{x^2}{2}}dx\nonumber \\
& = & \frac{J}{2^{l+1}\sqrt{2\pi}}\int_{-\infty}^{\infty}\frac{2}{\sqrt{2\pi}}\int_{x\geq 0} \lp 1+i\erfi\lp\frac{t}{\sqrt{2}}\rp\rp^{k-l_1} \lp 1-i\erfi\lp\frac{t}{\sqrt{2}}\rp\rp^{l-k+l_1} e^{-\frac{lt^2}{2}}e^{-itx} e^{-\frac{x^2}{2}}dxdt\nonumber \\
& = & \frac{J}{2^{l+1}\sqrt{2\pi}}\int_{-\infty}^{\infty} \lp1+i\erfi\lp\frac{t}{\sqrt{2}}\rp\rp^{k-l_1} \lp1-i\erfi\lp\frac{t}{\sqrt{2}}\rp\rp^{l-k+l_1+1} e^{-\frac{(l+1)t^2}{2}}dt.
\label{eq:posint1anal11}
\end{eqnarray}

\subsubsection{Computing $\phiint(0,F^{(l,l_2,bin)}_2)$}
\label{sec:posint2ang}

As observed in \cite{Stojnicl1RegPosfinn}, when computing $\phiint(0,F^{(l,l_2,bin)}_2)$ there is no need to change variables and find a new orthonormal basis. Instead analogously to (\ref{eq:posint1anal8}) we have
\begin{eqnarray}
\phiint(0,F^{(l,l_2,bin)}_2) &  = & \frac{1}{(2\pi)^{\frac{l}{2}}}\int_{-\1_{1\times l}\w_{l_1+1:l+l_1}\geq 0,\w_{l_1+1:k}\leq 0,\w_{k+1:l+l_1}\geq 0}e^{-\frac{\w_{l_1+1:l+l_1}^T\w_{l_1+1:l+l_1}}{2}}d\w_{l_1+1:l+l_1}\nonumber \\
& = & \frac{2^l}{2^l(2\pi)^{\frac{l}{2}}}\int_{-\1_{1\times l}\w_{l_1+1:l+l_1}\geq 0,\w_{l_1+1:k}\leq 0,\w_{k+1:l+l_1}\geq 0}e^{-\frac{\w_{l_1+1:l+l_1}^T\w_{l_1+1:l+l_1}}{2}}d\w_{l_1+1:l+l_1}\nonumber \\
&  = & \frac{1}{2^{l}}\int_{x\geq 0}f_{x^{(bin)}}(x)dx.\label{eq:posint2anal8}
\end{eqnarray}
Combining (\ref{eq:posint1anal10}) and (\ref{eq:posint2anal8}) we also have
\begin{eqnarray}
\phiint(0,F^{(l,l_2,bin)}_2) &  = &  \frac{1}{2^{l}}\int_{x\geq 0}f_{x^{(bin)}}(x)dx\nonumber \\
& = & \frac{1}{2^{l}}\int_{x\geq 0}\frac{1}{2\pi}\int_{-\infty}^{\infty}
\lp1+i\erfi\lp\frac{t}{\sqrt{2}}\rp\rp^{k-l_1}
\lp1-i\erfi\lp\frac{t}{\sqrt{2}}\rp\rp^{l-k+l_1} e^{-\frac{lt^2}{2}}e^{-itx}dt dx.\nonumber \\
\label{eq:posint2anal11}
\end{eqnarray}
Although (\ref{eq:posint2anal11}) can be used to compute $\phiint(0,F^{(l,l_2,bin)}_2)$ one can transform it a bit further
\begin{eqnarray}
\phiint(0,F^{(l,l_2,bin)}_2)
& = & \frac{1}{2^{l}}\int_{x\geq 0}\frac{1}{2\pi}\int_{-\infty}^{\infty}
\lp1+i\erfi\lp\frac{t}{\sqrt{2}}\rp\rp^{k-l_1}
\lp1-i\erfi\lp\frac{t}{\sqrt{2}}\rp\rp^{l-k+l_1} e^{-\frac{lt^2}{2}}e^{-itx}dt dx\nonumber \\
& = & \frac{1}{2^{l}}\frac{1}{2\pi}\lim_{x\rightarrow \infty}\lim_{\epsilon\rightarrow 0_+}(\int_{-\infty}^{-\epsilon} \lp1+i\erfi\lp\frac{t}{\sqrt{2}}\rp\rp^{k-l_1}\lp1-i\erfi\lp\frac{t}{\sqrt{2}}\rp\rp^{l-k+l_1} e^{-\frac{lt^2}{2}}\frac{(1-e^{-itx})}{it}dt\nonumber \\
& & +\int_{\epsilon}^{\infty} \lp1+i\erfi\lp\frac{t}{\sqrt{2}}\rp\rp^{k-l_1} \lp1-i\erfi\lp\frac{t}{\sqrt{2}}\rp\rp^{l-k+l_1} e^{-\frac{lt^2}{2}}\frac{(1-e^{-itx})}{it}dt).
\label{eq:posint2anal12}
\end{eqnarray}


\subsection{External angles}
\label{sec:posextang}

Computation of the external angles will also differ depending on the type of the face for which the angle is computed. We will therefore separately handle $\phiext(F^{(l,l_1,bin)}_1,C^{(bin)}_w)$ and $\phiext(F^{(l,l_2,bin)}_2,C^{(bin)}_w)$.

\subsubsection{Computing $\phiext(F^{(l,l_1,bin)}_1,C^{(bin)}_w)$}
\label{sec:posext1ang}

As was the case in \cite{Stojnicl1RegPosfinn} here, the external angles can also be handled through the GCOB strategy. Following \cite{Stojnicl1RegPosfinn} we first observe that for face $F^{(l,l_1,bin)}_1$ we have the corresponding positive hull of the outward normals to the meeting hyperplanes given by
\begin{equation}\label{eq:posext1anal1}
phull^{(l,l_1,bin)}_{ext,1}\triangleq  -pos(-e_{1},-e_{2},\dots,-e_{l_1},e_{l+l_1+2},e_{l+l_1+3},\dots,e_{n},-\1_{n\times 1}),
\end{equation}
where, as usual, $e_i\in \mR^n$ and its only nonzero component is at the $i$-th location and is equal to one ($pos(\cdot,\cdot,\dots,\cdot)$ obviously stands for the positive hull of the vectors inside the parenthesis; to facilitate writing we will work with $-phull^{(l,l_1,bin)}_{ext,1}$; due to obvious symmetry there is really no difference for our purposes here between $-phull^{(l,l_1,bin)}_{ext,1}$ and $phull^{(l,l_1,bin)}_{ext,1}$). To implement the GCOB strategy we define quantities in a way analogous to the way the corresponding quantities were defined in \cite{Stojnicl1RegPosfinn}. For the basis that we will work in we choose the columns of the following matrix
\begin{equation}
B^{(l,l_1,bin)}_{ext,1}=\begin{bmatrix}
-e_{1} & -e_{2} & \dots & -e_{l_1} & e_{l+l_1+2} &  e_{l+l_1+3} & \dots & e_{n} & \begin{bmatrix}
                                       \0_{l_1\times 1}\\
                                       -\1_{(l+1)\times 1} \\
                                       \0_{(n-l-1-l_1)\times 1}
                                     \end{bmatrix}\frac{1}{\sqrt{l+1}}
          \end{bmatrix}.\label{eq:posext1anal2}
\end{equation}
The subspace spanned by the columns of $B^{(l,l_1,bin)}_{ext,1}$ clearly contains $-phull^{(l,l_1,bin)}_{ext,1}$. To find a convenient parametrization
of $pos(-e_{1},-e_{2},\dots,-e_{l_1},e_{l+l_1+2},e_{l+l_1+3},\dots,e_{n},-\1_{n\times 1})$ we look at vectors $\g$ (clearly, this time $\g\in \mR^{n-l}$) such that
\begin{equation}
G^{(l,l_1,bin)}_{ext,1} =\{\g\in \mR^{n-l}| B^{(l,l_1,bin)}_{ext,1}\g\in pos(-e_{1},-e_{2},\dots,-e_{l_1},e_{l+l_1+2},e_{l+l_1+3},\dots,e_{n},-\1_{n\times 1})\triangleq-phull^{(l,l_1,bin)}_{ext,1} \}.\label{eq:posext1anal3}
\end{equation}
Setting
\begin{equation}
D^{(l,l_1,bin)}_{ext,1}=\begin{bmatrix}
-e_{1} & -e_{2} & \dots & -e_{l_1} & e_{l+l_1+2} &  e_{l+l_1+3} & \dots & e_{n} & -\1_{n\times 1}\end{bmatrix},\label{eq:posext1anal4}
\end{equation}
we can write
\begin{multline}
pos(-e_{1},-e_{2},\dots,-e_{l_1},e_{l+l_1+2},e_{l+l_1+3},\dots,e_{n},-\1_{n\times 1})  =  \{D^{(l,l_1,bin)}_{ext,1}\g^{(D)}|\g^{(D)}\geq 0,\g^{(D)}\in \mR^{n-l}\} \\
 =  \{\begin{bmatrix}
-e_{1} & -e_{2} & \dots & -e_{l_1} & e_{l+l_1+2} &  e_{l+l_1+3} & \dots & e_{n} & -\1_{n\times 1}\end{bmatrix}\g^{(D)}|\g^{(D)}\geq 0,\g^{(D)}\in \mR^{n-l}\}.\label{eq:posext1anal5}
\end{multline}
Combining (\ref{eq:posext1anal3}) and (\ref{eq:posext1anal5}) we obtain
\begin{equation}
G^{(l,l_1,bin)}_{ext,1} =\{\g\in \mR^{n-l}| \exists \g^{(D)}\in\mR^{n-l}, \g^{(D)}\geq 0, \quad \mbox{and} \quad B^{(l,l_1,bin)}_{ext,1}\g= D^{(l,l_1,bin)}_{ext,1}\g^{(D)} \}.\label{eq:posext1anal6}
\end{equation}
Similarly to what was the case in \cite{Stojnicl1RegPosfinn}, equations $(l_1+j),j=1,2,\dots,(l+1)$, in $B^{(l,+)}_{ext,1}\g= D_{ext,1}\g^{(D)}$ imply that
\begin{equation}
\g_{n-l}\frac{1}{\sqrt{l+1}}=\g^{(D)}_{n-l}\geq 0.\label{eq:posext1anal7}
\end{equation}
On the other hand, first $l_1$ and last $n-l-1-l_1$ equations in $B^{(l,+)}_{ext,1}\g= D_{ext,1}\g^{(D)}$ imply different conclusions when compared to \cite{Stojnicl1RegPosfinn} and what ultimately relates to the positive $\ell_1$. Namely, for $j\in\{1,2,\dots,l_1\}$, $j$-th equations in $B^{(l,l_1,bin)}_{ext,1}\g= D^{(l,l_1,bin)}_{ext,1}\g^{(D)}$ combined with (\ref{eq:posext1anal7}) imply that
\begin{eqnarray}
& & -\g_{j}=-\g^{(D)}_{j}-\g^{(D)}_{n-l}=-\g^{(D)}_{j}-\g_{n-l}\frac{1}{\sqrt{l+1}}\nonumber \\
\Longleftrightarrow & & -\g_{j}+\g_{n-l}\frac{1}{\sqrt{l+1}}=-\g^{(D)}_{j}\leq 0 \nonumber \\
\Longleftrightarrow & & \g_{j}\geq \g_{n-l}\frac{1}{\sqrt{l+1}}.
\label{eq:posext1anal8}
\end{eqnarray}
For $j\in\{l_1+1,l_1+2,\dots,n-l-1\}$, $l+j+1$-th equations in $B^{(l,l_1,bin)}_{ext,1}\g= D^{(l,l_1,bin)}_{ext,1}\g^{(D)}$ combined with (\ref{eq:posext1anal7}) imply that
\begin{eqnarray}
& & \g_{j}=\g^{(D)}_{j}-\g^{(D)}_{n-l}=\g^{(D)}_{j}-\g_{n-l}\frac{1}{\sqrt{l+1}}\nonumber \\
\Longleftrightarrow & & \g_{j}+\g_{n-l}\frac{1}{\sqrt{l+1}}=\g^{(D)}_{j}\geq 0 \nonumber \\
\Longleftrightarrow & & \g_{j}\geq -\g_{n-l}\frac{1}{\sqrt{l+1}}.
\label{eq:posext1anal8a}
\end{eqnarray}
A combination of (\ref{eq:posext1anal6}), (\ref{eq:posext1anal7}), (\ref{eq:posext1anal8}), and (\ref{eq:posext1anal8a}) then gives
\begin{equation}
G^{(l,l_1,bin)}_{ext,1} =\{\g\in \mR^{n-l}| \g_{n-l}\geq 0, \g_{j}\geq \frac{\g_{n-l}}{\sqrt{l+1}}, j\in\{1,2,\dots,l_1\},
\g_{j}\geq -\frac{\g_{n-l}}{\sqrt{l+1}}, j\in\{l_1+1,l_1+2,\dots,n\} \}.\label{eq:posext1anal9}
\end{equation}
Now that we have the above parametrization of the content of $-phull^{(l,l_1,bin)}_{ext,1}$ what is left to do to complete the GCOB is to work with independent standard normal, i.e. Gaussian, coordinates (as the GCOB states, due to the rotational invariance of the standard normal distribution working with them then automatically gives the fraction of the subspace taken by the set they define). To that end we finally have for $\phiext(F^{(l,l_1,bin)}_1,C^{(bin)}_w)$
\begin{eqnarray}
\phiext(F^{(l,l_1,bin)}_1,C^{(bin)}_w) &  = & \frac{1}{(2\pi)^{\frac{n-l}{2}}}\int_{\g\in G^{(l,l_1,bin)}_{ext,1}}e^{-\frac{\g^T\g}{2}}d\g \nonumber \\
& = & \frac{1}{(2\pi)^{\frac{l}{2}}}\int_{\g_{n-l}\geq 0}e^{-\frac{\g_{n-l}^2}{2}}
\lp \prod_{j=1}^{l_1}\frac{1}{(2\pi)^{\frac{l}{2}}}\int_{\g_{j}\geq \g_{n-l}\frac{1}{\sqrt{l+1}}}e^{-\frac{\g_{j}^2}{2}}d\g_{j}\rp\nonumber \\
& & \times
\lp \prod_{j=l_1+1}^{n-l-1}\frac{1}{(2\pi)^{\frac{l}{2}}}\int_{\g_{j}\geq -\g_{n-l}\frac{1}{\sqrt{l+1}}}e^{-\frac{\g_{j}^2}{2}}d\g_{j}\rp d\g_{n-l}\nonumber \\
& = & \frac{1}{(2\pi)^{\frac{l}{2}}}\int_{\g_{n-l}\geq 0}e^{-\frac{\g_{n-l}^2}{2}}
\lp \frac{1}{2}\erfc\lp\frac{\g_{n-l}}{\sqrt{2}\sqrt{l+1}}\rp\rp^{l_1}\nonumber \\
& & \times
\lp \frac{1}{2}\erfc\lp\frac{-\g_{n-l}}{\sqrt{2}\sqrt{l+1}}\rp\rp^{n-l-1-l_1} d\g_{n-l}.
\label{eq:posext1anal10}
\end{eqnarray}

\subsubsection{Computing $\phiext(F^{(l,l_2,bin)}_2,C^{(bin)}_w)$}
\label{sec:posext2ang}

Computing $\phiext(F^{(l,l_2,bin)}_2,C^{(bin)}_w)$ requires much less effort than the above computation of $\phiext(F^{(l,l_1,bin)}_1,C^{(bin)}_w)$. The positive hull of the outward normals for face $F^{(l,l_2,bin)}_2$ is given by
\begin{equation}\label{eq:posext2anal1}
phull^{(l,l_2,bin)}_{ext,2}\triangleq  -pos(-e_{1},-e_{2},\dots,-e_{l_2},e_{l+l_2+1},e_{l+l_2+2},\dots,e_{n}).
\end{equation}
One then automatically has
\begin{equation}
\phiext(F^{(l,l_2,bin)}_2,C^{(bin)}_w)  =  \frac{1}{2^{n-l}}.
\label{eq:posext2anal2}
\end{equation}
All of what we presented above is then enough to determine $p^{(bin)}_{err}(k,m,n)$. We summarize the obtained results in the following theorem.
\begin{theorem}(Exact \emph{binary} $\ell_1$'s performance characterization -- finite dimensions)
Let $A$ be an $m\times n$ matrix in (\ref{eq:system})
with i.i.d. standard normal components (or, alternatively, with the null-space uniformly distributed in the Grassmanian). Let
the unknown $\x$ in (\ref{eq:system}) be binary $k$-sparse. Let $p^{(bin)}_{err}(k,m,n)$ be the probability that the solutions of (\ref{eq:binl0}) and (\ref{eq:l1}) do \emph{not} coincide. Then
\begin{eqnarray}
p^{(bin)}_{err}(k,m,n) & = & 2 ( \sum_{l=m+2j+1,j\in \mN_0}^{n-1} \sum_{l_1=l^{(min)}_1}^{l^{(max)}_1} c^{(l,l_1,bin)}_1\phiint(0,F^{(l,l_1,bin)}_1)\phiext(F^{(l,l_1,bin)}_1,C^{(bin)}_w)\nonumber \\
& & + \sum_{l=m+2j+1,j\in \mN_0}^{n} \sum_{l_2=l^{(min)}_2}^{l^{(max)}_2} c^{(l,l_2,bin)}_2 \phiint(0,F^{(l,l_2,bin)}_2)\phiext(F^{(l,l_2,bin)}_2,C^{(bin)}_w)),
\label{eq:posfinalthm1}
\end{eqnarray}
where $l^{(min)}_1$, $l^{(max)}_1$, $l^{(min)}_2$, and $l^{(max)}_2$ are as given in (\ref{eq:posintanal1}), and $c^{(l,l_1,bin)}_1$, $ c^{(l,l_2,bin)}_2$, $\phiint(0,F^{(l,l_1,bin)}_1)$, $\phiint(0,F^{(l,l_2,bin)}_2)$, $\phiext(F^{(l,l_1,bin)}_1,C^{(bin)}_w)$, and $\phiext(F^{(l,l_2,bin)}_2,C^{(bin)}_w)$ are as given in (\ref{eq:posintanal4}), (\ref{eq:posintanal5}), (\ref{eq:posint1anal11}), (\ref{eq:posint2anal12}), (\ref{eq:posext1anal10}), and  (\ref{eq:posext2anal2}), respectively.
\label{thm:posfinalperr}
\end{theorem}

\subsection{Simulations and theoretical results -- binary $\ell_1$}
\label{sec:posthnumresults}

Theorem \ref{thm:posfinalperr} provides a powerful tool to determine $p^{(bin)}_{err}(k,m,n)$. Here we present the concrete values that we obtained for $p^{(bin)}_{err}(k,m,n)$ utilizing Theorem \ref{thm:posfinalperr}. Namely, in Figure \ref{fig:l1regnonperr} we show first separately on the left graphed $p^{(bin)}_{err}(k,m,n)$ values. These were obtained by fixing $k=5$ and $n=30$ and varying $m$ within the so-called transition zone so that $p^{(bin)}_{err}(k,m,n)$ changes from one to zero. In the same figure on the right we also show the values that we obtained through numerical simulations for the same sets of system's dimensions $(k,m,n)$. We complement Figure \ref{fig:l1regnonperr} by adding Table \ref{tab:l1regnonperrtab1} where we show for several triplets $(k,m,n)$ the numerical values for $p^{(bin)}_{err}(k,m,n)$ that were graphed in Figure \ref{fig:l1regnonperr}. On top of that, Table \ref{tab:l1regnonperrtab1} also contains for these same $(k,m,n)$ triplets the number of numerical experiments that were run as well as the number of them that did not result in having the solution of (\ref{eq:l1}) match the binary solution of (\ref{eq:binl0}). It is not that hard to see from both, Figure \ref{fig:l1regnonperr} and Table \ref{tab:l1regnonperrtab1}, that there is an overwhelming agreement between the simulated and the theoretical results.

\begin{figure}[htb]
\begin{minipage}[b]{.5\linewidth}
\centering
\centerline{\epsfig{figure=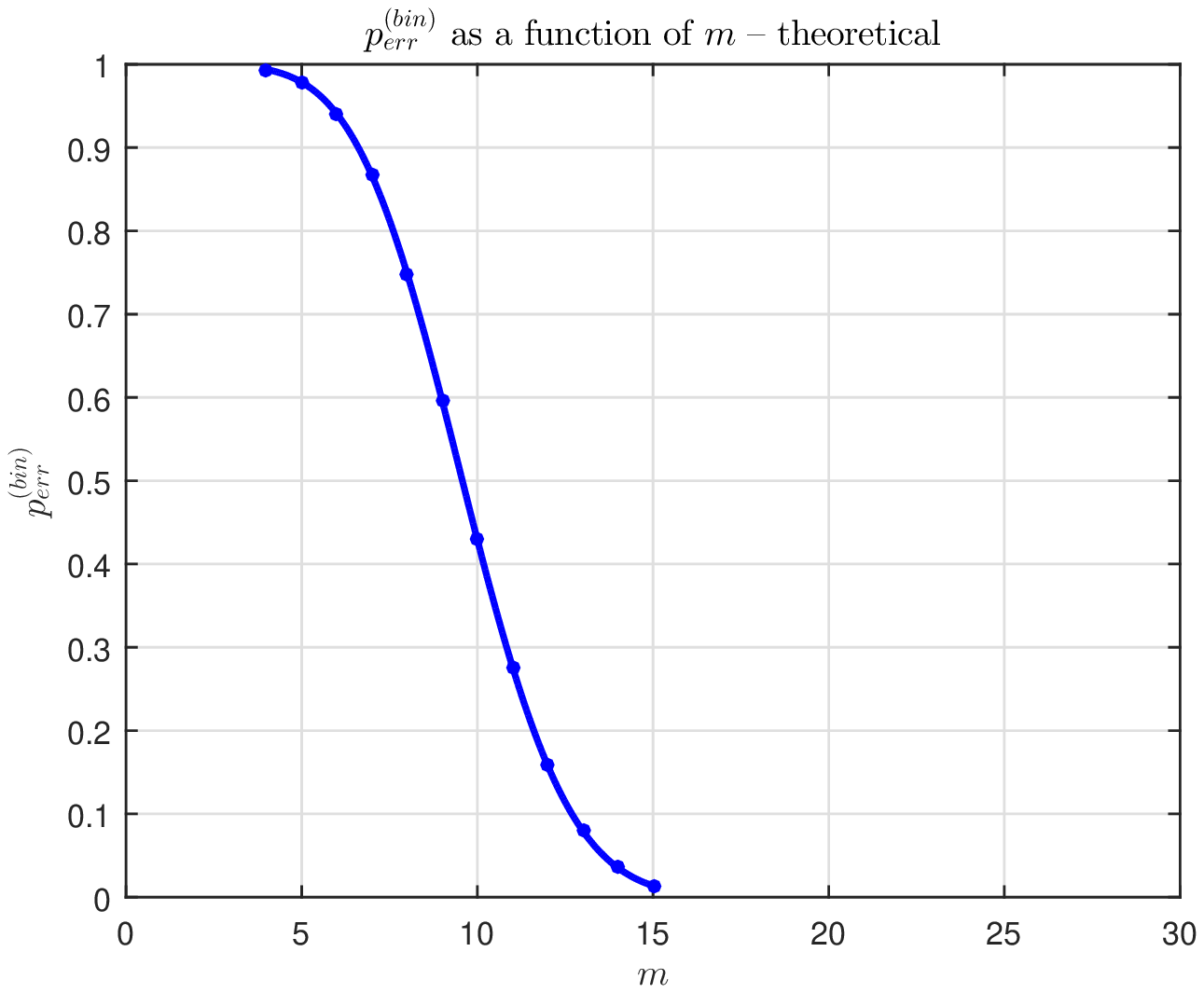,width=9cm,height=7cm}}
\end{minipage}
\begin{minipage}[b]{.5\linewidth}
\centering
\centerline{\epsfig{figure=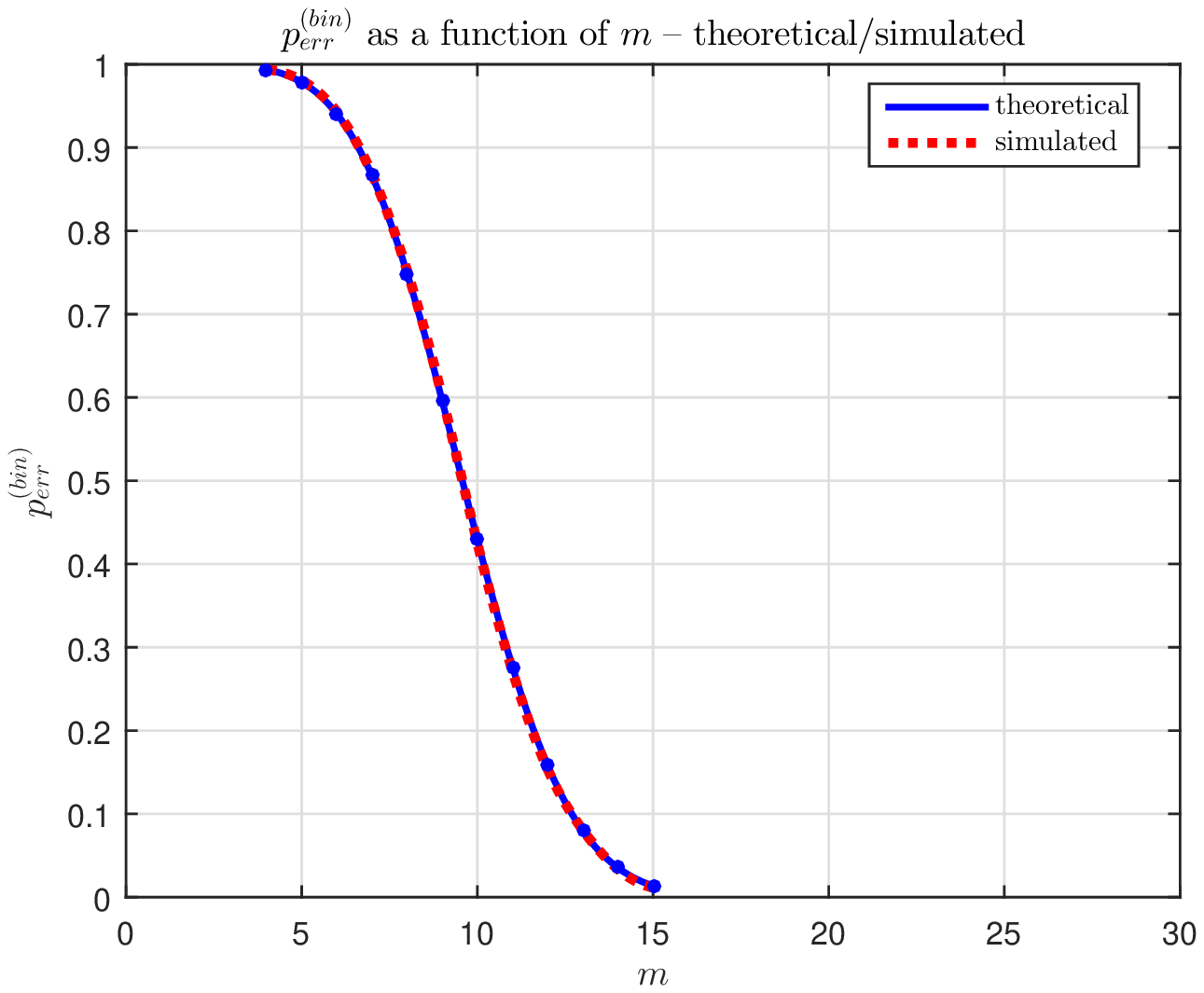,width=9cm,height=7cm}}
\end{minipage}
\caption{$p^{(bin)}_{err}(k,m,n)$ as a function of $m$ ($k=5$, $n=30$); left -- theory; right -- simulations}
\label{fig:l1regnonperr}
\end{figure}

\begin{table}[h]
\caption{Simulated and theoretical results for $p^{(bin)}_{err}(k,m,n)$; $k=5$, $n=30$}\vspace{.1in}
\hspace{-0in}\centering
\begin{tabular}{||c||c|c|c|c|c|c||}\hline\hline
$m$                            & $ 7 $ & $ 8 $ & $ 9 $ & $ 10 $ & $ 11 $ & $ 12 $ \\ \hline \hline
$\#$ of failures               & $ 5255 $ & $ 4515 $ & $ 3198 $ & $ 2265 $ & $ 2328 $ & $ 1303 $ \\ \hline \hline
$\#$ of repetitions            & $ 6051 $ & $ 6002 $ & $ 5321 $ & $ 5326 $ & $ 8573 $ & $ 8476 $ \\ \hline \hline
$p^{(bin)}_{err}(k,m,n)$ -- simulation& $ \bl{\mathbf{0.8685}} $ & $ \bl{\mathbf{0.7522}} $ & $ \bl{\mathbf{0.6010}} $ & $ \bl{\mathbf{0.4253}} $ & $ \bl{\mathbf{0.2716}} $ & $ \bl{\mathbf{0.1537}} $ \\ \hline \hline
$p^{(bin)}_{err}(k,m,n)$ -- theory    & $ \mathbf{0.8663} $ & $ \mathbf{0.7489} $ & $ \mathbf{0.5958} $ & $ \mathbf{0.4292} $ & $ \mathbf{0.2766} $ & $ \mathbf{0.1582} $ \\ \hline \hline
\end{tabular}
\label{tab:l1regnonperrtab1}
\end{table}

\subsection{Asymptotics}
\label{sec:posasym}

As mentioned earlier on several occasions, the results that we presented in previous sections can be thought of as finite dimensional complements to the asymptotic ones presented in companion paper \cite{Stojnicl1BnBxasymldp}. Although the main interest in this paper is study of finite dimensional behavior, in this section we will sketch how the results derived here behave as the dimensions grow large (of course, a full detailed analysis of the asymptotic regime we leave for \cite{Stojnicl1BnBxasymldp}; here we just present what will eventually be some of the starting points in \cite{Stojnicl1BnBxasymldp}.

As is by now well known, increasing systems' dimensions (say assuming their linear proportionality) the so-called phase-transition phenomenon emerges (this was observed and characterized through \cite{DonohoPol,DonohoUnsigned,StojnicCSetam09,StojnicUpper10}) for the standard $\ell_1$ and in \cite{StojnicISIT2010binary} for the binary $\ell_1$). In other words, for any given $\beta$ the transition zone (where $p^{(bin)}_{err}(k,m,n)$ changes from one to zero) basically shrinks to a single point, say $\alpha^{(bin)}_w$, such that for $\alpha>\alpha^{(bin)}_w$ $p^{(bin)}_{err}(k,m,n)$ goes to zero at an exponential rate and for $\alpha<\alpha^{(bin)}_w$ $p^{(bin)}_{cor}(k,m,n)=1-p^{(bin)}_{err}(k,m,n)$ goes to zero at an exponential rate. Our companion paper \cite{Stojnicl1BnBxasymldp} manages to settle the so-called LDP phenomenon which basically assumes a precise determination of the rates of decay of these probabilities. As is obvious from \cite{Stojnicl1BnBxasymldp}, the LDP analysis is not super trivial at all and goes way beyond what we discuss here. However, some of the starting points of the analyses presented in \cite{Stojnicl1BnBxasymldp} can be connected to some of the results proven in earlier sections in this paper and below we show that connection.

Since we will be switching to the asymptotic regime, until the end of this subsection we will assume the above mentioned so-called linear regime, i.e. we will assume that as $n$ is getting large, $k=\beta n$ and $m=\alpha n$, where $\beta$ and $\alpha$ are fixed constants independent of $n$. Also, we find it convenient to set $\rho\triangleq \lim_{n\rightarrow\infty}\frac{l}{n}$, $\rho_1\triangleq \lim_{n\rightarrow\infty}\frac{l_1}{n}$, and $\rho_2\triangleq \lim_{n\rightarrow\infty}\frac{l_2}{n}$. Also, let $S_{\rho_1}=S_{\rho_2}=(\max(0,\beta-\rho),\min(\beta,1-\rho))$ (these definitions are for the completeness; it is rather clear though that only $\rho\geq \beta$ is of interest here). Then, as $n\rightarrow\infty$, from (\ref{eq:posfinalthm1}) we have
\begin{equation}
\lim_{n\rightarrow \infty}\frac{\log(p^{(bin)}_{err}(k,m,n))}{n}  =  \max\{\max_{\rho\geq \alpha,\rho_1\in S_{\rho_1}}\lim_{n\rightarrow \infty} \frac{\log(\zeta^{(\infty,bin)}_1)}{n},\max_{\rho\geq \alpha,\rho_2\in S_{\rho_2}}\lim_{n\rightarrow \infty}\frac{\log(\zeta^{(\infty,bin)}_2)}{n}\},
\label{eq:posasym1}
\end{equation}
where
\begin{eqnarray} \label{eq:posasym2}
  \lim_{n\rightarrow \infty} \frac{\log(\zeta^{(\infty,bin)}_1)}{n} &=&
  \lim_{n\rightarrow \infty} \lp\frac{\log(c^{(l,l_1,bin)}_{1})}{n}+
  \frac{\log(\phiint(0,F^{(l,l_1,bin)}_1))}{n}+\frac{\log(\phiext(F^{(l,l_1,bin)}_1,C^{(bin)}_w))}{n}\rp \nonumber \\
  \lim_{n\rightarrow \infty} \frac{\log(\zeta^{(\infty,bin)}_2)}{n} &=&
  \lim_{n\rightarrow \infty} \lp\frac{\log(c^{(l,l_2,bin)}_{2})}{n}+
  \frac{\log(\phiint(0,F^{(l,l_2,bin)}_2))}{n}+\frac{\log(\phiext(F^{(l,l_2,bin)}_2,C^{(bin)}_w))}{n}\rp.\nonumber \\
\end{eqnarray}
From (\ref{eq:posintanal4}) and (\ref{eq:posintanal5}) we have
\begin{eqnarray}
\lim_{n\rightarrow \infty} \frac{\log(c^{(l,l_1,bin)}_{1})}{n}=\lim_{n\rightarrow \infty}\frac{\log\binom{k}{l_1}\binom{n-k}{n-l-1-l_1}}{n}=-\beta H\lp\frac{\rho_1}{\beta}\rp
-(1-\beta)H\lp\frac{1-\rho-\rho_1}{1-\beta}\rp,\label{eq:posasym3a}
\end{eqnarray}
and
\begin{eqnarray}
\lim_{n\rightarrow \infty} \frac{\log(c^{(l,l_2,bin)}_{2})}{n}=\lim_{n\rightarrow \infty}\frac{\log\binom{k}{l_2}\binom{n-k}{n-l-l_2}}{n}=-\beta H\lp\frac{\rho_2}{\beta}\rp-(1-\beta)H\lp\frac{1-\rho-\rho_2}{1-\beta}\rp,\label{eq:posasym3b}
\end{eqnarray}
where
\begin{equation}\label{eq:posasym4}
  H(x)=x\log(x)+(1-x)\log(1-x).
\end{equation}
Focusing on $l_2=l_1$ and following \cite{Stojnicl1RegPosfinn}, from (\ref{eq:posint1anal8}) and (\ref{eq:posint2anal8}) we have
\begin{eqnarray}
\lim_{n\rightarrow \infty} \frac{\log(\phiint(0,F^{(l,l_1,bin)}_1)}{n} & \leq & \lim_{n\rightarrow \infty} \frac{\log(\phiint(0,F^{(l,l_2,bin)}_2)}{n}.\label{eq:posasym4a}
\end{eqnarray}
Let
\begin{equation}\label{eq:posasym4b}
  c_F(l)=\sqrt{\frac{\sqrt{l+1}-1}{l\sqrt{l+1}}},
\end{equation}
and
\begin{equation}\label{eq:posasym4c}
B_F^{(l,l_1,bin)}=\begin{bmatrix}
  \0_{l_1\times l_1}  &   \0_{l_1\times (l+1)} & \0_{l_1\times (n-l-1-l_1)}\\
\0_{(l+1) \times l_1}  &
  \begin{bmatrix}
      I_{l\times l}-c_F(l)^2\1_{l\times 1}\1_{1\times l} & -\1_{l\times 1}\frac{1}{\sqrt{l+1}} & \\
        -\1_{1\times l}\frac{1}{\sqrt{l+1}} & -\frac{1}{\sqrt{l+1}}
        \end{bmatrix} & \0_{(l+1)\times (n-l-1-l_1)} \\
 \0_{(n-l-1-l_1) \times l_1}  &        \0_{(n-l-1-l_1)\times (l+1)} & \0_{(n-l-1-l_1)\times (n-l-1-l_1)}
      \end{bmatrix}.
\end{equation}
Clearly, $(B_F^{(l,l_1,bin)})^TB_F^{(l,l_1,bin)}=I_{n\times n}$ which implies
\begin{equation}\label{eq:posasym4ca}
\phiint(0,F^{(l,l_2,bin)}_2)=\phiint(0,B_F^{(l,l_1,bin)} F^{(l,l_2,bin)}_2).
\end{equation}
Recall that
\begin{eqnarray}\label{eq:posasym4d}
F^{(l,l_1,bin)}_1 & = & \{\w\in \mR^n|-\sum_{i=l_1+1}^k \w_i = \sum_{i=k+1}^{l+l_1+1}\w_{i},\w_{l_1+1:k}\leq 0,\w_{1:l_1}= 0,
\w_{k+1:l+l_1+1}\geq 0,\w_{l+l_1+2:n}= 0\} \nonumber \\
F^{(l,l_2,bin)}_2 & = & \{\w\in \mR^n|-\sum_{i=1}^k \w_i\geq \sum_{i=k+1}^{l}\w_{i},,\w_{l_2+1:k}\leq 0,\w_{1:l_2}= 0,
\w_{k+1:l+l_2}\geq 0,\w_{l+l_2+1:n}= 0\}.
\end{eqnarray}
One can now repeat line by line the arguments from \cite{Stojnicl1BnBxasymldp} to conclude
\begin{equation}\label{eq:posasym4h}
\phiint(0,F^{(l,l_2,bin)}_2)=\phiint(0,F^{(l,l_1,bin)}_2)=\phiint(0,B_F^{(l,l_1,bin)} F^{(l,l_1,bin)}_2) \leq \phiint(0,F^{(l,l_1,bin)}_1).
\end{equation}
Finally, combining (\ref{eq:posasym4a}) and (\ref{eq:posasym4h}) we obtain
\begin{eqnarray}
\lim_{n\rightarrow \infty} \frac{\log(\phiint(0,F^{(l,l_1,bin)}_1)}{n}  =  \lim_{n\rightarrow \infty} \frac{\log(\phiint(0,F^{(l,l_2,bin)}_2)}{n}.
\label{eq:posasym4i}
\end{eqnarray}
From (\ref{eq:posint2anal8}) we also have
\begin{eqnarray}
\phiint(0,F^{(l,l_1,bin)}_2) &  = &   \frac{2^l}{2^l(2\pi)^{\frac{l}{2}}}\int_{-\1_{1\times l}\w_{l_1+1:l+l_1}\geq 0,\w_{l_1+1:k}\leq 0,\w_{k+1:l+l_1}\geq 0}e^{-\frac{\w_{l_1+1:l+l_1}^T\w_{l_1+1:l+l_1}}{2}}d\w_{l_1+1:l+l_1}\nonumber \\
&  = & \frac{1}{2^{l}}P(-\1_{1\times l}\w_{l_1+1:l+l_1}\geq 0),\label{eq:posasym6}
\end{eqnarray}
where on the right side of the last equality one can think of the elements of $\w_{l_1+1:k}$ as being the i.i.d. negative standard half normals and the elements of $\w_{k+1:l+l_1}$ as being the i.i.d. standard half normals. By the definition of the large deviations principle we further have
\begin{eqnarray}
\lim_{n\rightarrow \infty} \frac{\log(\phiint(0,F^{(l,l_1,bin)}_2))}{n}
&  = & \lim_{n\rightarrow \infty} \frac{\log(\frac{1}{2^{l}}P(-\1_{1\times l}\w_{l_1+1:l+l_1}\geq 0))}{n}\nonumber \\
& = &  \min_{\mu_y\geq 0} \lim_{n\rightarrow \infty} \frac{\log(\mE e^{-\mu_y\1_{1\times l}\w_{l_+1:l+l_1}})}{n}-\rho\log(2)\nonumber \\
& = &  \min_{\mu_y\geq 0} \lp(\rho+\rho_1-\beta) \log\lp\mE e^{-\mu_y\w_{k+1}}\rp+(\beta-\rho_1)\log\lp\mE e^{\mu_y\w_{l_1+1}}\rp\rp-\rho\log(2) \nonumber \\
& = &  \min_{\mu_y\geq 0} ((\rho+\rho_1-\beta) \log\lp\frac{2}{\sqrt{2\pi}}\int_{0}^{\infty} e^{-\frac{\w_{k+1}^2}{2}-\mu_y\w_{k+1}}d\w_{k+1}\rp \nonumber \\
& & +(\beta-\rho_1) \log\lp\frac{2}{\sqrt{2\pi}}\int_{0}^{\infty} e^{-\frac{\w_{l_1+1}^2}{2}+\mu_y\w_{l_1+1}}d\w_{l_1+1}\rp)-\rho\log(2) \nonumber \\
& = &  \min_{\mu_y\geq 0} \lp(\rho+\rho_1-\beta) \log\lp\erfc\lp\frac{\mu_y}{\sqrt{2}}\rp\rp+(\beta-\rho_1)\log\lp\erfc\lp\frac{-
\mu_y}{\sqrt{2}}\rp\rp +\rho\frac{\mu_y^2}{2}\rp\nonumber \\
& & -\rho\log(2) \nonumber \\
& = &  \min_{\mu_y\geq 0} \lp(\rho+\rho_1-\beta) \log(\erfc(\mu_y))+(\beta-\rho_1)\log(\erfc(-\mu_y))+\rho\frac{\mu_y^2}{2}\rp-\rho\log(2).\nonumber \\\label{eq:posasym7}
\end{eqnarray}
From (\ref{eq:posext1anal10}) we obtain
\begin{multline}
\lim_{n\rightarrow \infty} \frac{\log(\phiext(F^{(l,l_1,bin)}_1,C^{(bin)}_w))}{n} \\
 = \lim_{n\rightarrow \infty} \frac{\log\lp\frac{1}{(2\pi)^{\frac{l}{2}}}\int_{\g_{n-l}\geq 0}e^{-\frac{\g_{n-l}^2}{2}}
\lp \frac{1}{2}\erfc\lp\frac{\g_{n-l}}{\sqrt{2}\sqrt{l+1}}\rp\rp^{l_1}
\lp \frac{1}{2}\erfc\lp\frac{-\g_{n-l}}{\sqrt{2}\sqrt{l+1}}\rp\rp^{n-l-1-l_1} d\g_{n-l}\rp}{n}  \\
 =  \max_{\g_{n-l}\geq 0} \lp -\frac{\g_{n-l}^2}{2}+(1-\rho-\rho_1)\log\lp \frac{1}{2}\erfc\lp\frac{-\g_{n-l}}{\sqrt{2}\sqrt{\rho}}\rp\rp
 +\rho_1\log\lp \frac{1}{2}\erfc\lp\frac{\g_{n-l}}{\sqrt{2}\sqrt{\rho}}\rp\rp \rp \\
 =  \max_{\g_{n-l}\geq 0} \lp -\rho \g_{n-l}^2+(1-\rho-\rho_1)\log\lp \frac{1}{2}\erfc(-\g_{n-l})\rp
 +\rho_1\log\lp \frac{1}{2}\erfc(\g_{n-l})\rp \rp,
\label{eq:posasym8}
\end{multline}
and from (\ref{eq:posext2anal2}) we have
\begin{eqnarray}
\lim_{n\rightarrow \infty} \frac{\log(\phiext(F^{(l,l_2,bin)}_2,C^{(bin)}_w))}{n}
 =  \lim_{n\rightarrow \infty} \frac{\log\lp\frac{1}{2^{n-l}}\rp}{n}=-(1-\rho)\log(2).
\label{eq:posasym9}
\end{eqnarray}
For $\g_{n-l}=0$ in (\ref{eq:posasym8}) we have $\lim_{n\rightarrow \infty} \frac{\log(\phiext(F^{(l,l_1,bin)}_1,C^{(bin)}_w))}{n}=-(1-\rho)\log(2)$ which implies that  \begin{eqnarray}
\lim_{n\rightarrow \infty} \frac{\log(\phiext(F^{(l,l_2,bin)}_2,C^{(bin)}_w))}{n}
 \leq  \lim_{n\rightarrow \infty} \frac{\log(\phiext(F^{(l,l_1,bin)}_1,C^{(bin)}_w))}{n}.
\label{eq:posasym10}
\end{eqnarray}
A combination of (\ref{eq:posasym1}), (\ref{eq:posasym2}), (\ref{eq:posasym3a}), (\ref{eq:posasym3b}), (\ref{eq:posasym4i}), and (\ref{eq:posasym10}) gives
\begin{eqnarray}
\lim_{n\rightarrow \infty}\frac{\log(p^{(bin)}_{err}(k,m,n))}{n} & = & \max\{\max_{\rho\geq \alpha,\rho_1\in S_{\rho_1}}\lim_{n\rightarrow \infty} \frac{\log(\zeta^{(\infty,bin)}_1)}{n},\max_{\rho\geq \alpha}\lim_{n\rightarrow \infty}\frac{\log(\zeta^{(\infty,bin)}_2)}{n}\}\nonumber \\
& = & \max_{\rho\geq \alpha,\rho_1\in S_{\rho_1}}\lim_{n\rightarrow \infty} \frac{\log(\zeta^{(\infty,bin)}_1)}{n}.
\label{eq:posasym11}
\end{eqnarray}
Relying on (\ref{eq:posasym2}), (\ref{eq:posasym3a}), (\ref{eq:posasym3b}), (\ref{eq:posasym7}), and (\ref{eq:posasym8}), one can rewrite (\ref{eq:posasym11}) in the following way
\begin{eqnarray}
\lim_{n\rightarrow \infty}\frac{\log(p^{(bin)}_{err}(k,m,n))}{n}
 & = & \max_{\rho\geq \alpha,\rho_1\in S_{\rho_1}}\lim_{n\rightarrow \infty} \frac{\log(\zeta^{(\infty,bin)}_1)}{n}\nonumber \\
 & = & \max_{\rho\geq \alpha,\rho_1\in S_{\rho_1}} (-\beta H\lp\frac{\rho_1}{\beta}\rp
-(1-\beta)H\lp\frac{1-\rho-\rho_1}{1-\beta}\rp \nonumber \\
 & & + \min_{\mu_y\geq 0} \lp(\rho+\rho_1-\beta) \log(\erfc(\mu_y))+(\beta-\rho_1)\log(\erfc(-\mu_y))+\rho\frac{\mu_y^2}{2}\rp-\rho\log(2)\nonumber \\.
& & + \max_{\g_{n-l}\geq 0} \lp -\rho \g_{n-l}^2+(1-\rho-\rho_1)\log\lp \frac{1}{2}\erfc(-\g_{n-l})\rp
 +\rho_1\log\lp \frac{1}{2}\erfc(\g_{n-l})\rp \rp.\nonumber \\
\label{eq:posasym12}
\end{eqnarray}
Now, for any given $\beta$, let $\alpha^{(bin)}_w$ be the $\alpha$ that ensures $\lim_{n\rightarrow \infty}\frac{\log(p^{(bin)}_{err}(k,m,n))}{n} =0$ (as \cite{Stojnicl1BnBxasymldp} proves, such an $\alpha$ always exists). This is exactly the so-called phase-transition value of $\alpha$. One can now also follow line by line the reasoning in \cite{Stojnicl1RegPosfinn} regarding the optimal $\rho$ to conclude that if if $\alpha\geq \alpha^{(bin)}_w$ it is equal to $\alpha$. This means that the outer optimization in (\ref{eq:posasym12}) can be removed. Moreover, if $\alpha\leq \alpha^{(bin)}_w$ the optimal $\rho$ will again trivially be equal to $\alpha$. The only difference is that now one looks at the following complementary version of (\ref{eq:posanal3})
\begin{equation}
1-2\sum_{l=m-2j-1,j\in \mN_0,l\geq k-1} \sum_{F^{(l,bin)}\in \calF^{(l,bin)}}\phiint(0,F^{(l,bin)})\phiext(F^{(l,bin)},C^{(bin)}_w)=1-p^{(bin)}_{cor},\label{eq:posasym13}
\end{equation}
where $p^{(bin)}_{cor}$ is the probability of being correct, i.e. the probability that the solution of (\ref{eq:l1}) is the binary sparse solution of (\ref{eq:binl0}) and its decay rate is given by (\ref{eq:posasym12}) with $\rho\leq \alpha$. \cite{Stojnicl1RegPosfinn} uses the same line of reasoning applied for the probability of error to draw the conclusions about the optimality of $\rho$ in the context of the probability of being correct. That line of reasoning applies here as well and we have that again the optimal $\rho$ is indeed equal to $\alpha$ and that the outer optimization in (\ref{eq:posasym12}) can be removed. (\ref{eq:posasym12}) is then sufficient to fully determine numerically PT and LDP curves of the binary $\ell_1$. \cite{Stojnicl1BnBxasymldp}, goes way beyond that and explicitly solves (\ref{eq:posasym12}).

\section{Box $\ell_1$}
\label{sec:l1}

In this section we discuss the use of the modified $\ell_1$ from (\ref{eq:l1}) in the recovery of the box-constrained sparse vectors, i.e. in solving (\ref{eq:boxl0}). To ensure that everything is properly defined we will for a moment revisit the definition of the box-constrained sparse vectors and make it a bit more specific. Namely, in Section \ref{sec:back}, we introduced box-constrained $k$-sparse vector $\x$ as a vector that has $k$ components strictly inside the interval $[0,1]$ and remaining $(n-k)$ components at the edges of the interval. That of course will continue to be valid here. However, now we will also add that there are precisely $\kmu=(1-\mu)(n-k)$ ($\mu\in[\frac{1}{2},1)$) components of $\x$ that are equal to one and precisely $(n-k-\kmu)$ components that are equal to zero. Moreover, without loss of generality we will assume that the vector $\x$ that is the solution of (\ref{eq:boxl0}) is such that its elements $\x_{\kmu+k+1},\x_{k+\kmu+2},\dots,\x_{n}$ are equal to zero, its elements $\x_{1},\x_{2},\dots,\x_{\kmu}$ are equal to one, and its elements $\x_{\kmu+1},\x_{\kmu+2},\dots,\x_{\kmu+k}$ are strictly inside the interval $[0,1]$. Also, we should add that if one would like to draw a parallel with the standard sparse vectors, in a way both zero and one components here play the role of the zero components of the standard sparse vectors (for that reason we may often refer to the components of a box-constrained vector that belong to $(0,1)$ as the nonzero components, although technically box-constrained vectors have components that are equal to one which are also not zeros). The following result is established in \cite{Stojnicl1BnBxasymldp}
and is basically a box-constrained adaptation of a result proven for the general $\ell_1$ in \cite{StojnicCSetam09,StojnicICASSP09} and the binary $\ell_1$ in \cite{StojnicISIT2010binary}.
\begin{theorem}(\cite{StojnicCSetam09,StojnicICASSP09,StojnicISIT2010binary} ``Nonzero" elements of box-constrained $\x$ have fixed location)
Assume that an $m\times n$ system matrix $A$ is given. Let $\x$
be a box-constrained $k$-sparse vector and let $\mu$ be a real number such that $\mu\in[\frac{1}{2},1]$ and $\kmu=(1-\mu)(n-k)$ is an integer. Also, let each element of $\x$ belong to $[0,1]$ interval and let $\x_{\kmu+k+1}=\x_{\kmu+k+2}=\dots=\x_{n}=0$ and $\x_{1}=\x_{2}=\dots=\x_{\kmu}=1$. Further, assume that $\y\triangleq A\x$ and that $\w$ is
an $n\times 1$ vector such that $\w_i\leq 0,1\leq i\leq \kmu$, and $\w_i\geq 0,\kmu+k+1\leq i\leq n$. If
\begin{equation}
(\forall \w\in \textbf{R}^n | A\w=0) \quad  -\sum_{i=\kmu+1}^{\kmu+k} \w_i<\sum_{i=1}^{\kmu}\w_{i}+\sum_{i=\kmu+k+1}^{n}\w_{i},
\end{equation}
then the solutions of (\ref{eq:boxl0}) and (\ref{eq:l1}) coincide. Moreover, if
\begin{equation}
(\exists \w\in \textbf{R}^n | A\w=0) \quad  -\sum_{i=\kmu+1}^{\kmu+k} \w_i\geq \sum_{i=1}^{\kmu}\w_{i}+\sum_{i=\kmu+k+1}^{n}\w_{i},
\label{eq:boxthmeqgen}
\end{equation}
then there will be a box-constrained $\x$ from the above set such that it is the solution of (\ref{eq:boxl0}) and is not the solution of (\ref{eq:l1}).
\label{thm:thmregposcond}
\end{theorem}
\begin{proof}
  Follows by a couple of simple modifications of the arguments leading up to Theorem \ref{thm:posthmregposcond}.
\end{proof}

As was done in Section \ref{sec:posl1}, we will below try to follow the strategy presented in \cite{Stojnicl1RegPosfinn}. Here though, many of the adjustments already made in Section \ref{sec:posl1} will be utilized as well. As usual, we will try to emphasize any new additional adjustments and avoid the repetitions of the already made ones. Now, to exploit the results given in the above theorem, we again start by setting
\begin{equation}
C^{(box)}_w\triangleq\{\w\in \mR^n| \quad -\sum_{i=\kmu+1}^{\kmu+k} \w_i\geq \sum_{i=1}^{\kmu}\w_{i}+\sum_{i=\kmu+k+1}^{n}\w_{i},
\w_i\geq 0,\kmu+k+1\leq i\leq n,\w_i\leq 0,1\leq i\leq \kmu\}.\label{eq:defSw}
\end{equation}
Clearly, $C^{(box)}_w$ is a polyhedral cone and analogously to (\ref{eq:posanal3}) we have
\begin{equation}
p^{(box)}_{err}(k,m,n,\kmu)=2\sum_{l=m+2j+1,j\in \mN_0}^{n} \sum_{F^{(l,box)}\in \calF^{(l,box)}}\phiint(0,F^{(l,box)})\phiext(F^{(l,box)},C^{(box)}_w),\label{eq:anal3}
\end{equation}
where $p^{(box)}_{err}(k,m,n,\kmu)$ is the probability that the solution of (\ref{eq:l1}) is not the binary $k$-sparse solution of (\ref{eq:binl0}), $F^{(l,box)}$ is an $l$-face of $C^{(box)}_w$, and $\calF^{(l,box)}$ is the set of all $l$-faces of $C^{(box)}_w$ (as earlier, $\phiint(\cdot,\cdot)$ and $\phiext(\cdot,\cdot)$ are the so-called internal and external angles). To be able to use (\ref{eq:anal3}) to compute $p^{(box)}_{err}(k,m,n,\kmu)$ one needs to compute the corresponding $\phiint(\cdot,\cdot)$ and $\phiext(\cdot,\cdot)$ as well. Below we show how it can be done.

\subsection{Internal angles}
\label{sec:intang}

In this section we handle $\phiint(0,F^{(l,box)})$. As in Section \ref{sec:posl1}, we split the set of all $l$-faces $\calF^{(l,box)}$ into two sets, $\calF^{(l,box)}_1$ and $\calF^{(l,box)}_2$ in the following way. Set
\begin{eqnarray}
I_l & \triangleq  & \{1,2,\dots,\kmu\} \nonumber \\
I_r & \triangleq  & \{\kmu+k+1,\kmu+k+2,\dots,n\} \nonumber \\
l_1^{(min)} & \triangleq & \max(0,n-l-1-(n-k-\kmu)) \nonumber \\
l_1^{(max)} & \triangleq & \min(\kmu-1,n-l-1) \nonumber \\
l_2^{(min)} & \triangleq & \max(0,n-l-(n-k-\kmu)) \nonumber \\
l_2^{(max)} & \triangleq & \min(\kmu-1,n-l),\label{eq:intanal1}
\end{eqnarray}
and write
\begin{equation}
\calF^{(l,box)}_1 \triangleq \cup_{l_1=l_1^{(min)}}^{l_1^{(max)}}\calF^{(l,l_1,box)}_1, l\in \{0,1,\dots,n-1\},\label{eq:intanal1a}
\end{equation}
where
\begin{multline}
\calF^{(l,l_1,box)}_1 \triangleq\{\w\in \mR^n| \quad -\sum_{i=\kmu+1}^{\kmu+k} \w_i= \sum_{i=1}^{\kmu}\w_{i}+\sum_{i=\kmu+k+1}^{n}\w_{i},\w_{I_l}\leq 0,\w_{I_r}\geq 0, \\
\w_{I^{(l,l_1,box)}_{1,r}}=0,I^{(l,l_1,box)}_{1,r}\subset I_r,|I^{(l,l_1,box)}_{1,r}|=n-l-1-l_1 \quad \mbox{and} \quad \w_{I^{(l,l_1,box)}_{1,l}}=0,I^{(l,l_1,box)}_{1,l}\subset I_l,|I^{(l,l_1,box)}_{1,l}|=l_1\},\label{eq:intanal2}
\end{multline}
and
\begin{equation}
\calF^{(l,box)}_2 \triangleq \cup_{l_2=l_2^{(min)}}^{l_2^{(max)}}\calF^{(l,l_2,box)}_2, l\in \{1,2,\dots,n\},\label{eq:intanal2a}
\end{equation}
where
\begin{multline}
\calF^{(l,l_2,box)}_1 \triangleq\{\w\in \mR^n| \quad -\sum_{i=\kmu+1}^{\kmu+k} \w_i\geq \sum_{i=1}^{\kmu}\w_{i}+\sum_{i=\kmu+k+1}^{n}\w_{i},\w_{I_l}\leq 0,\w_{I_r}\geq 0, \\
\w_{I^{(l,l_2,box)}_{2,r}}=0,I^{(l,l_2,box)}_{2,r}\subset I_r,|I^{(l,l_2,box)}_{2,r}|=n-l-l_1 \quad \mbox{and} \quad \w_{I^{(l,l_2,box)}_{2,l}}=0,I^{(l,l_2,box)}_{2,l}\subset I_l,|I^{(l,l_2,box)}_{2,l}|=l_1\}.\label{eq:intanal3}
\end{multline}
Simple combinatorial considerations give the following cardinalities of sets $\calF^{(l,l_1,box)}_1$ and $\calF^{(l,l_1bin)}_2$
\begin{equation}
c^{(l,l_1,box)}_1\triangleq |\calF^{(l,l_1,box)}_1|=\binom{\kmu}{l_1}\binom{n-k-\kmu}{n-l-1-l_1},\label{eq:intanal4}
\end{equation}
and
\begin{equation}
c^{(l,l_2,box)}_2\triangleq |\calF^{(l,l_2,box)}_2|=\binom{\kmu}{l_2}\binom{n-k-\kmu}{n-l-l_2}.\label{eq:intanal5}
\end{equation}
All of the above helps transform (\ref{eq:anal3}) so that it can be written in an way analogous to (\ref{eq:posintanal6})
\begin{eqnarray}
p^{(box)}_{err}(k,m,n,\kmu) & = & 2\sum_{l=m+2j+1,j\in \mN_0}^{n} ( \sum_{F^{(l,box)}_1\in \calF^{(l,box)}_1}\phiint(0,F^{(l,box)}_1)\phiext(F^{(l,box)}_1,C^{(box)}_w)\nonumber \\
& & + \sum_{F^{(l,box)}_2\in \calF^{(l,box)}_2}\phiint(0,F^{(l,box)}_2)\phiext(F^{(l,box)}_2,C^{(box)}_w))\nonumber \\
& = & 2 ( \sum_{l=m+2j+1,j\in \mN_0}^{n-1} \sum_{l_1=l^{(min)}_1}^{l^{(max)}_1} c^{(l,l_1,box)}_1\phiint(0,F^{(l,l_1,box)}_1)\phiext(F^{(l,l_1,box)}_1,C^{(box)}_w) \nonumber \\
& & + \sum_{l=m+2j+1,j\in \mN_0}^{n} \sum_{l_2=l^{(min)}_2}^{l^{(max)}_2} c^{(l,l_2,box)}_2 \phiint(0,F^{(l,l_2,box)}_2)\phiext(F^{(l,l_2,box)}_2,C^{(box)}_w)),\nonumber \\
\label{eq:intanal6}
\end{eqnarray}
Relying on symmtery as in Section \ref{sec:posl1}
(and ultimately \cite{Stojnicl1RegPosfinn}) we can basically choose $F^{(l,l_1,box)}_1$ and $F^{(l,l_2,box)}_2$ as any of the elements from sets $\calF^{(l,l_1,box)}_1$ and $\calF^{(l,l_2,box)}_2$, respectively. For the concreteness we choose
\begin{eqnarray}
  I^{(l,l_1,box)}_{1,l} & = & \{1,2,\dots,l_1\} \nonumber\\
  I^{(l,l_1,box)}_{1,r} & = & \{l+l_1+2,l+l_1+3,\dots,n\} \nonumber\\
  I^{(l,l_2,box)}_{2,l} & = & \{1,2,\dots,l_2\} \nonumber\\
  I^{(l,l_2,box)}_{2,r} & = & \{l+l_2+1,l+l_2+2,\dots,n\},\label{eq:intanal6a}
\end{eqnarray}
and consequently
\begin{multline}
F^{(l,l_1,box)}_1 =\{\w\in \mR^n| \quad  -\sum_{i=\kmu+1}^{\kmu+k}\w_i=\sum_{i=l_1+1}^{\kmu} \w_{i}+\sum_{i=\kmu+k+1}^{n}\w_{i},\\
\w_{l_1+1:\kmu}\leq 0,\w_{1:l_1}=0,\w_{\kmu+k+1:l+l_1+1}\geq 0,\w_{l+l_1+2:n}=0\},\label{eq:intanal7}
\end{multline}
and
\begin{multline}
F^{(l,l_2,box)}_2 =\{\w\in \mR^n| \quad  -\sum_{i=\kmu+1}^{\kmu+k}\w_i\geq \sum_{i=l_1+1}^{\kmu} \w_{i}+\sum_{i=\kmu+k+1}^{n}\w_{i},\\
\w_{l_2+1:\kmu}\leq 0,\w_{1:l_2}=0,\w_{\kmu+k+1:l+l_2}\geq 0,\w_{l+l_2+1:n}=0\},\label{eq:intanal8}
\end{multline}
In the following subsections we will determine separately $\phiint(0,F^{(l,l_1,box)}_1)$ and $\phiint(0,F^{(l,l_2,box)}_2)$.

\subsubsection{Computing $\phiint(0,F^{(l,l_1,box)}_1)$}
\label{sec:int1ang}

As in Section \ref{sec:posl1}, we will rely on the GCOB strategy. Moreover, we carefully indexed everything so that we can utilize the columns of $B^{(l,l_1,box)}_{int,1}=B^{(l,l_1,bin)}_{int,1}$ as an orthonormal basis. What is left to do is to determine the set of $\g$'s (say, $G^{(l,l_1,box)}_1$) such that $B^{(l,l_1,box)}_{int,1}\g\in F^{(l,l_1,box)}_1$. In other words we are looking for
\begin{equation}
G^{(l,l_1,box)}_1 =\{\g\in \mR^l| \w=B^{(l,l_1,box)}_{int,1}\g \quad \mbox{and}\quad \w_{l_1+1:k}\leq 0,\w_{k+1:l+l_1+1}\geq 0\}.\label{eq:int1anal2}
\end{equation}
Relying on the structure of $B^{(l,l_1,box)}_{int,1}$, (\ref{eq:int1anal2}) can be written in the following equivalent form as well
\begin{equation}
G^{(l,l_1,box)}_1 =\{\g\in \mR^l| \w_{l_1+1:(l+l_1+1)}=\begin{bmatrix}
            \begin{bmatrix}
              B \\
              \0_{1\times (l-1)}
            \end{bmatrix}  & \begin{bmatrix}
              -\1_{l\times 1} \\
              l
            \end{bmatrix}\frac{1}{\sqrt{l^2+l}}
          \end{bmatrix}\g \quad \mbox{and}\quad \w_{l_1+1:k}\leq 0,\w_{\kmu+k+1:l+l_1+1}\geq 0\}.\label{eq:int1anal3}
\end{equation}
As in Section \ref{sec:posl1}, we can now repeat step by step of the remaining considerations from \cite{Stojnicl1RegPosfinn} to finally arrive at the following characterization of $\phiint(0,F^{(l,l_1,box)}_1)$
\begin{eqnarray}
\phiint(0,F^{(l,l_1,box)}_1) &  = & \frac{J}{(2\pi)^{\frac{l}{2}}}\int\limits_{\substack{-\1_{1\times l}\w_{l_1+1:l+l_1}\geq 0,\w_{l_1+1:\kmu}\leq 0,\\\w_{\kmu+k+1:l+l_1}\geq 0}}
e^{-\frac{\|\w_{l_1+1:l+l_1}\|^2+
             ((-\1_{1\times l})\w_{l_1+1:l+l_1})^2}{2}}d\w_{l_1+1:l+l_1}\nonumber \\
&  = & \frac{J2^{l-k}}{2^{l-k}(2\pi)^{\frac{l}{2}}}\int\limits_{\substack{-\1_{1\times l}\w_{l_1+1:l+l_1}\geq 0,\w_{l_1+1:\kmu}\leq 0,\\\w_{\kmu+k+1:l+l_1}\geq 0}}
e^{-\frac{\|\w_{l_1+1:l+l_1}\|^2+
             ((-\1_{1\times l})\w_{l_1+1:l+l_1})^2}{2}}d\w_{l_1+1:l+l_1}\nonumber \\
&  = & \frac{J}{2^{l-k}}\int_{x\geq 0}f_{x^{(box)}}(x)e^{-\frac{x^2}{2}}dx,\label{eq:int1anal8}
\end{eqnarray}
where $J=\sqrt{l+1}$ and $f_{x^{(box)}}(\cdot)$ is the pdf of the random variable $x^{(box)}=-\1_{1\times l}\w_{l_1+1:l+l_1}$, where $\w_i,l_1+1 \leq i\leq \kmu$, are  i.i.d. negative standard half-normals, $\w_i,\kmu +1 \leq i\leq \kmu+k$, are  i.i.d. standard normals, and $\w_i,k+1 \leq i\leq l+l_1$, are  i.i.d. standard half-normals. Following the strategy of Section \ref{sec:posl1}, to determine $f_{x^{(box)}}(\cdot)$, we will first compute the characteristic function of $x^{(box)}$. We have
\begin{eqnarray}
E e^{itx^{(box)}} & = & \frac{2^{l}}{(2\pi)^{\frac{l}{2}}}\int_{\w_{l_1+1:\kmu}\leq 0,\w_{\kmu+k+1:l+l_1}\geq 0}e^{-\frac{\w_{l_1+1:l+l_1}^T\w_{l_1+1:l+l_1}}{2}-it\1_{1\times l}\w_{l_1+1:l+l_1}}d\w_{l_1+1:l+l_1} \nonumber \\
& = & \lp 1+i\erfi\lp\frac{t}{\sqrt{2}}\rp\rp^{\kmu-l_1} \lp 1-i\erfi\lp\frac{t}{\sqrt{2}}\rp\rp^{l-k-\kmu+l_1} e^{-\frac{lt^2}{2}},\label{eq:int1anal9}
\end{eqnarray}
and
\begin{equation}
f_{x^{(box)}}(x)=\frac{1}{2\pi}\int_{-\infty}^{\infty} \lp 1+i\erfi\lp\frac{t}{\sqrt{2}}\rp\rp^{\kmu-l_1} \lp 1-i\erfi\lp\frac{t}{\sqrt{2}}\rp\rp^{l-k-\kmu+l_1} e^{-\frac{lt^2}{2}}e^{-itx}dt,\label{eq:int1anal10}
\end{equation}
where as earlier
\begin{equation}
\erfi(y)=-i\erf(iy)=\frac{2}{\sqrt{\pi}}\int_{0}^{y}e^{\frac{z^2}{2}}dz.\label{eq:int1anal10a}
\end{equation}
A combination of (\ref{eq:int1anal8}) and (\ref{eq:int1anal10}) finally gives
\begin{eqnarray}
\phiint(0,F^{(l,l_1,box)}_1) &  = &  \frac{J}{2^{l}}\int_{x\geq 0}f_{x^{(box)}}(x)e^{-\frac{x^2}{2}}dx\nonumber \\
& = & \frac{J}{2^{l-k}}\int_{x\geq 0}\frac{1}{2\pi}\int_{-\infty}^{\infty}
\lp 1+i\erfi\lp\frac{t}{\sqrt{2}}\rp\rp^{\kmu-l_1} \nonumber \\
& & \times
\lp 1-i\erfi\lp\frac{t}{\sqrt{2}}\rp\rp^{l-k-\kmu+l_1} e^{-\frac{lt^2}{2}}e^{-itx}dt e^{-\frac{x^2}{2}}dx\nonumber \\
& = & \frac{J}{2^{l-k+1}\sqrt{2\pi}}\int_{-\infty}^{\infty}\frac{2}{\sqrt{2\pi}}\int_{x\geq 0} \lp 1+i\erfi\lp\frac{t}{\sqrt{2}}\rp\rp^{\kmu-l_1} \nonumber \\
& & \times
\lp 1-i\erfi\lp\frac{t}{\sqrt{2}}\rp\rp^{l-k-\kmu+l_1} e^{-\frac{lt^2}{2}}e^{-itx} e^{-\frac{x^2}{2}}dxdt\nonumber \\
& = & \frac{J}{2^{l-k+1}\sqrt{2\pi}}\int_{-\infty}^{\infty} \lp1+i\erfi\lp\frac{t}{\sqrt{2}}\rp\rp^{\kmu-l_1} \lp1-i\erfi\lp\frac{t}{\sqrt{2}}\rp\rp^{l-k-\kmu+l_1+1} e^{-\frac{(l+1)t^2}{2}}dt.\nonumber \\
\label{eq:int1anal11}
\end{eqnarray}

\subsubsection{Computing $\phiint(0,F^{(l,l_2,box)}_2)$}
\label{sec:int2ang}

Following what was done in Section \ref{sec:posint2ang} we have for $\phiint(0,F^{(l,l_2,box)}_2)$
\begin{eqnarray}
\phiint(0,F^{(l,l_2,box)}_2) &  = & \frac{1}{(2\pi)^{\frac{l}{2}}}\int_{-\1_{1\times l}\w_{l_1+1:l+l_1}\geq 0,\w_{l_1+1:\kmu}\leq 0,\w_{\kmu+k+1:l+l_1}\geq 0}e^{-\frac{\w_{l_1+1:l+l_1}^T\w_{l_1+1:l+l_1}}{2}}d\w_{l_1+1:l+l_1}\nonumber \\
& = & \frac{2^{l-k}}{2^{l-k}(2\pi)^{\frac{l}{2}}}\int_{-\1_{1\times l}\w_{l_1+1:l+l_1}\geq 0,\w_{l_1+1:\kmu}\leq 0,\w_{\kmu+k+1:l+l_1}\geq 0}e^{-\frac{\w_{l_1+1:l+l_1}^T\w_{l_1+1:l+l_1}}{2}}d\w_{l_1+1:l+l_1}\nonumber \\
&  = & \frac{1}{2^{l-k}}\int_{x\geq 0}f_{x^{(box)}}(x)dx.\label{eq:int2anal8}
\end{eqnarray}
From (\ref{eq:int1anal10}) and (\ref{eq:int2anal8}) we obtain
\begin{eqnarray}
\phiint(0,F^{(l,l_2,box)}_2) &  = &  \frac{1}{2^{l-k}}\int_{x\geq 0}f_{x^{(box)}}(x)dx\nonumber \\
& = & \frac{1}{2^{l-k}}\int_{x\geq 0}\frac{1}{2\pi}\int_{-\infty}^{\infty}
\lp1+i\erfi\lp\frac{t}{\sqrt{2}}\rp\rp^{\kmu-l_1}
\lp1-i\erfi\lp\frac{t}{\sqrt{2}}\rp\rp^{l-k-\kmu+l_1} e^{-\frac{lt^2}{2}}e^{-itx}dt dx.\nonumber \\
\label{eq:int2anal11}
\end{eqnarray}
A bit of further transformation also gives
\begin{eqnarray}
\phiint(0,F^{(l,l_2,box)}_2)
& = & \frac{1}{2^{l-k}}\int_{x\geq 0}\frac{1}{2\pi}\int_{-\infty}^{\infty}
\lp1+i\erfi\lp\frac{t}{\sqrt{2}}\rp\rp^{\kmu-l_1}
\lp1-i\erfi\lp\frac{t}{\sqrt{2}}\rp\rp^{l-k-\kmu+l_1} e^{-\frac{lt^2}{2}}e^{-itx}dt dx\nonumber \\
& = & \frac{1}{2^{l-k}}\frac{1}{2\pi}\times \nonumber \\
& & \lim_{x\rightarrow \infty}\lim_{\epsilon\rightarrow 0_+}(\int_{-\infty}^{-\epsilon} \lp1+i\erfi\lp\frac{t}{\sqrt{2}}\rp\rp^{\kmu-l_1}\lp1-i\erfi\lp\frac{t}{\sqrt{2}}\rp\rp^{l-k-\kmu+l_1} e^{-\frac{lt^2}{2}}\frac{(1-e^{-itx})}{it}dt\nonumber \\
& & +\int_{\epsilon}^{\infty} \lp1+i\erfi\lp\frac{t}{\sqrt{2}}\rp\rp^{\kmu-l_1} \lp1-i\erfi\lp\frac{t}{\sqrt{2}}\rp\rp^{l-k-\kmu+l_1} e^{-\frac{lt^2}{2}}\frac{(1-e^{-itx})}{it}dt).
\label{eq:int2anal12}
\end{eqnarray}


\subsection{External angles}
\label{sec:extang}

Similarly to what we had for the internal angles, there are two types of the external angles of interest, $\phiext(F^{(l,l_1,box)}_1,C^{(box)}_w)$ and $\phiext(F^{(l,l_2,box)}_2,C^{(box)}_w)$. However, to compute these angles one can utilize already computed external angles for the binary case. Namely,
note that for face $F^{(l,l_1,box)}_1$ the positive hull of the outward normals to the meeting hyperplanes is given by
\begin{equation}\label{eq:ext1anal1}
phull^{(l,l_1,box)}_{ext,1}\triangleq  -pos(-e_{1},-e_{2},\dots,-e_{l_1},e_{l+l_1+2},e_{l+l_1+3},\dots,e_{n},-\1_{n\times 1})=phull^{(l,l_1,bin)}_{ext,1}.
\end{equation}
This basically immediately implies
\begin{eqnarray}
\phiext(F^{(l,l_1,box)}_1,C^{(box)}_w) &  = & \phiext(F^{(l,l_1,box)}_1,C^{(box)}_w) \nonumber \\
& = & \frac{1}{(2\pi)^{\frac{l}{2}}}\int_{\g_{n-l}\geq 0}e^{-\frac{\g_{n-l}^2}{2}}
\lp \frac{1}{2}\erfc\lp\frac{\g_{n-l}}{\sqrt{2}\sqrt{l+1}}\rp\rp^{l_1}
\lp \frac{1}{2}\erfc\lp\frac{-\g_{n-l}}{\sqrt{2}\sqrt{l+1}}\rp\rp^{n-l-1-l_1} d\g_{n-l}. \nonumber \\
\label{eq:ext1anal10}
\end{eqnarray}
Similarly we have that the positive hull of the outward normals for face $F^{(l,l_2,box)}_2$ is given by
\begin{equation}\label{eq:ext2anal1}
phull^{(l,l_2,box)}_{ext,2}\triangleq  -pos(-e_{1},-e_{2},\dots,-e_{l_2},e_{l+l_2+1},e_{l+l_2+2},\dots,e_{n})=phull^{(l,l_2,bin)}_{ext,2}.
\end{equation}
One then automatically has
\begin{equation}
\phiext(F^{(l,l_2,box)}_2,C^{(box)}_w) =\phiext(F^{(l,l_2,bin)}_2,C^{(bin)}_w) =  \frac{1}{2^{n-l}}.
\label{eq:ext2anal2}
\end{equation}
What we presented above is then enough to determine $p^{(box)}_{err}(k,m,n,\kmu)$. The following theorem summarizes the obtained results.
\begin{theorem}(Exact \emph{box} $\ell_1$'s performance characterization -- finite dimensions)
Let $A$ be an $m\times n$ matrix in (\ref{eq:system})
with i.i.d. standard normal components (or, alternatively, with the null-space uniformly distributed in the Grassmanian). Let
the unknown $\x$ in (\ref{eq:system}) be box-constrained $k$-sparse with $\kmu$ of its components equal to one. Let $p^{(box)}_{err}(k,m,n,\kmu)$ be the probability that the solutions of (\ref{eq:boxl0}) and (\ref{eq:l1}) do \emph{not} coincide. Then
\begin{eqnarray}
p^{(box)}_{err}(k,m,n,\kmu) & = & 2 ( \sum_{l=m+2j+1,j\in \mN_0}^{n-1} \sum_{l_1=l^{(min)}_1}^{l^{(max)}_1} c^{(l,l_1,box)}_1\phiint(0,F^{(l,l_1,box)}_1)\phiext(F^{(l,l_1,box)}_1,C^{(box)}_w)\nonumber \\
& & + \sum_{l=m+2j+1,j\in \mN_0}^{n} \sum_{l_2=l^{(min)}_2}^{l^{(max)}_2} c^{(l,l_2,box)}_2 \phiint(0,F^{(l,l_2,box)}_2)\phiext(F^{(l,l_2,box)}_2,C^{(box)}_w)),
\label{eq:finalthm1}
\end{eqnarray}
where $l^{(min)}_1$, $l^{(max)}_1$, $l^{(min)}_2$, and $l^{(max)}_2$ are as given in (\ref{eq:intanal1}), and $c^{(l,l_1,box)}_1$, $ c^{(l,l_2,box)}_2$, $\phiint(0,F^{(l,l_1,box)}_1)$, $\phiint(0,F^{(l,l_2,box)}_2)$, $\phiext(F^{(l,l_1,box)}_1,C^{(box)}_w)$, and $\phiext(F^{(l,l_2,box)}_2,C^{(box)}_w)$ are as given in (\ref{eq:intanal4}), (\ref{eq:intanal5}), (\ref{eq:int1anal11}), (\ref{eq:int2anal12}), (\ref{eq:ext1anal10}), and  (\ref{eq:ext2anal2}), respectively.
\label{thm:finalperr}
\end{theorem}

\subsection{Simulations and theoretical results -- box $\ell_1$}
\label{sec:boxthnumresults}

What Theorem \ref{thm:posfinalperr} was for computing $p^{(bin)}_{err}(k,m,n,\kmu)$ that is now
Theorem \ref{thm:finalperr} for computing $p^{(box)}_{err}(k,m,n,\kmu)$. The concrete values that we obtained for $p^{(box)}_{err}(k,m,n,\kmu)$ utilizing Theorem \ref{thm:finalperr} are shown in Figure \ref{fig:boxl1regnonperr} and Table \ref{tab:boxl1regnonperrtab1}. In addition to that we also show the corresponding results obtained for the same set of systems parameters through numerical simulations. For all results shown in Figure \ref{fig:boxl1regnonperr} and Table \ref{tab:boxl1regnonperrtab1} we fixed $k=5$, $n=30$, and $\kmu=5$, and varied $m$ within the transition zone so that $p^{(box)}_{err}(k,m,n,\kmu)$ changes from one to zero. As in Section \ref{sec:posl1}, in Table \ref{tab:boxl1regnonperrtab1} we also present the number of numerical experiments that were run as well as the number of them that did not result in having the solution of (\ref{eq:l1}) match the box-constrained solution of (\ref{eq:boxl0}). We again observe an overwhelming agreement between the simulated and the theoretical results.

\begin{figure}[htb]
\begin{minipage}[b]{.5\linewidth}
\centering
\centerline{\epsfig{figure=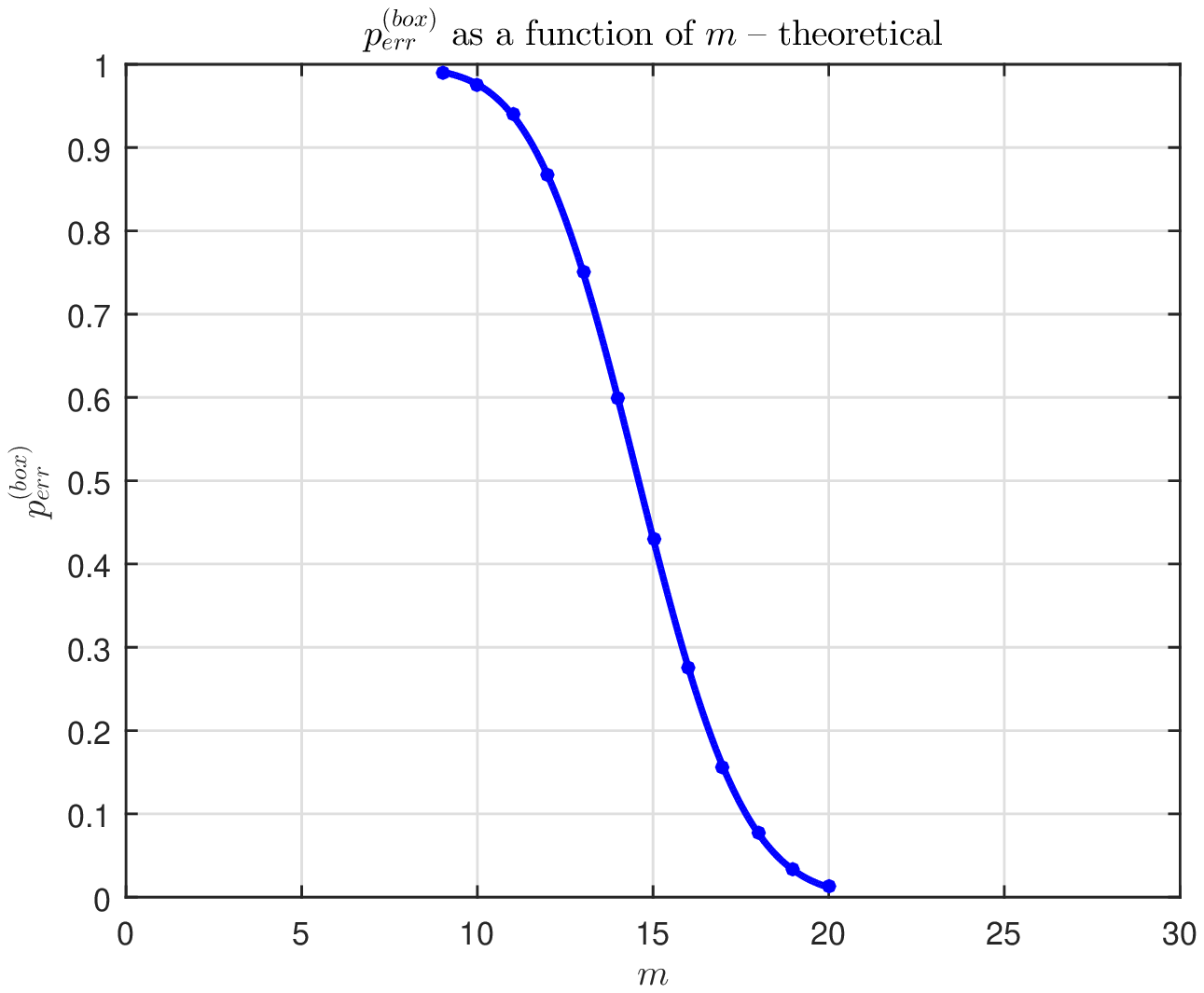,width=9cm,height=7cm}}
\end{minipage}
\begin{minipage}[b]{.5\linewidth}
\centering
\centerline{\epsfig{figure=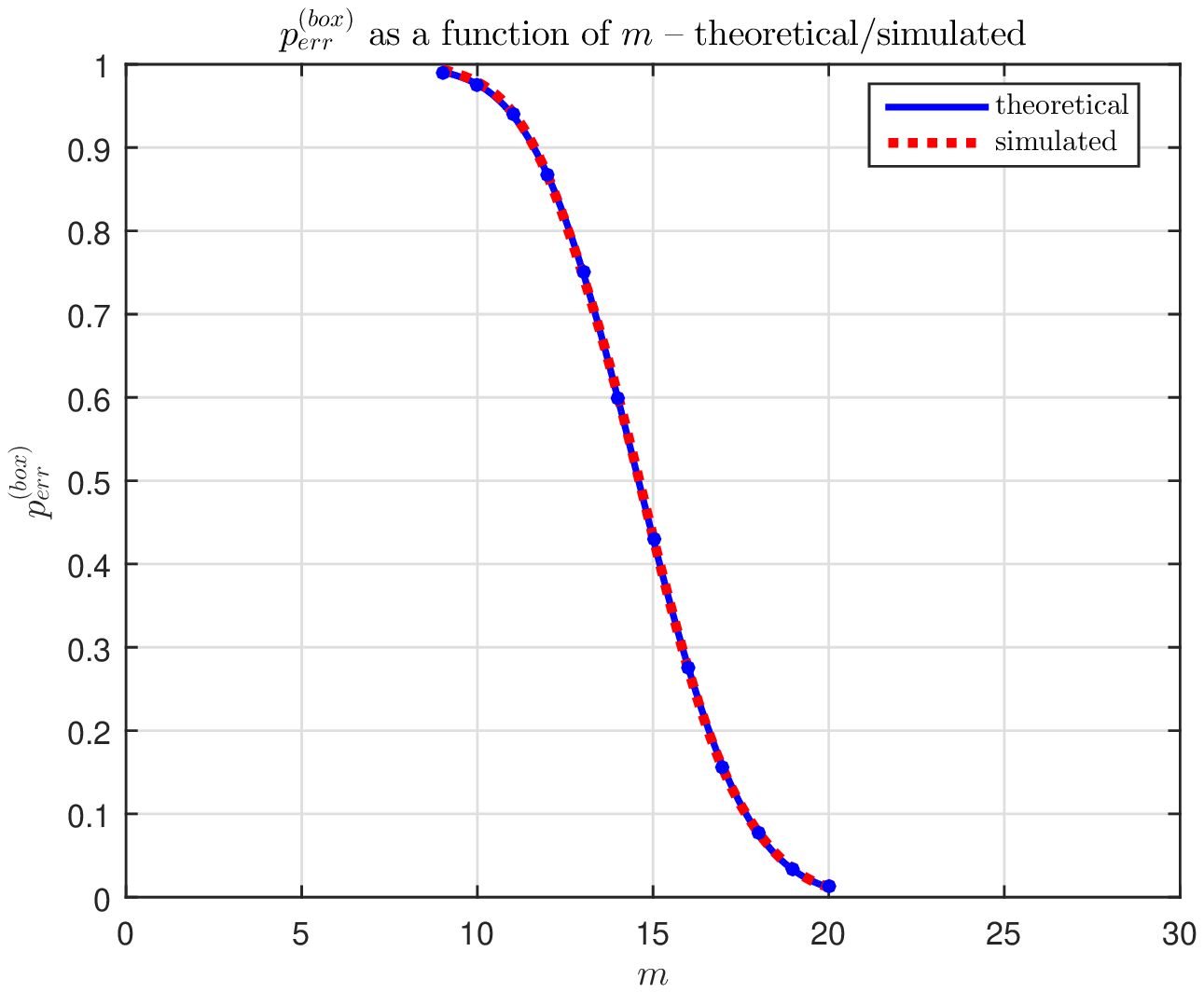,width=9cm,height=7cm}}
\end{minipage}
\caption{$p^{(box)}_{err}(k,m,n,\kmu)$ as a function of $m$ ($k=5$, $n=30$, $\kmu=5$); left -- theory; right -- simulations}
\label{fig:boxl1regnonperr}
\end{figure}

\begin{table}[h]
\caption{Simulated and theoretical results for $p^{(box)}_{err}(k,m,n,\kmu)$; $k=5$, $n=30$, $\kmu=5$}\vspace{.1in}
\hspace{-0in}\centering
\begin{tabular}{||c||c|c|c|c|c|c||}\hline\hline
$m$                            & $ 12 $ & $ 13 $ & $ 14 $ & $ 15 $ & $ 16 $ & $ 17 $ \\ \hline \hline
$\#$ of failures               & $ 10935 $ & $ 9371 $ & $ 7682 $ & $ 7802 $ & $ 4948 $ & $ 2776 $ \\ \hline
$\#$ of repetitions            & $ 12624 $ & $ 12537 $ & $ 12754 $ & $ 18009 $ & $ 18039 $ & $ 18116 $ \\ \hline \hline
$p^{(box)}_{err}$ -- simulation& $ \bl{\mathbf{0.8662}} $ & $ \bl{\mathbf{0.7475}} $ & $ \bl{\mathbf{0.6023}} $ & $ \bl{\mathbf{0.4332}} $ & $  \bl{\mathbf{0.2743}} $ & $ \bl{\mathbf{0.1532}} $ \\ \hline \hline
$p^{(box)}_{err}$ -- theory    & $ \mathbf{0.8668} $ & $ \mathbf{0.7512} $ & $ \mathbf{0.5988} $ & $ \mathbf{0.4312} $ & $ \mathbf{0.2764} $ & $ \mathbf{0.1558} $ \\ \hline \hline
\end{tabular}
\label{tab:boxl1regnonperrtab1}
\end{table}

\subsection{Asymptotics}
\label{sec:asym}

As we have done in Section \ref{sec:posasym} for the binary $\ell_1$, in this section we provide a bridge between the above finite dimensional considerations and the corresponding asymptotic ones from \cite{Stojnicl1BnBxasymldp} for the box $\ell_1$. Since the box $\ell_1$ also exhibits the phase-transition behavior, for any $\beta$ there will be a $\alpha^{(box)}_w$, such that for $\alpha>\alpha^{(box)}_w$ $p^{(box)}_{err}(k,m,n,\kmu)$ goes to zero at an exponential rate and for $\alpha<\alpha^{(box)}_w$ $p^{(box)}_{cor}(k,m,n,\kmu)=1-p^{(box)}_{err}(k,m,n,\kmu)$ goes to zero at an exponential rate. \cite{Stojnicl1BnBxasymldp}, among other things, settled the box $\ell_1$'s LDP phenomenon and precisely determined the rates of decay of these probabilities. Below we show how the above finite dimensional considerations can be connected to what is presented in an asymptotic scenario in \cite{Stojnicl1BnBxasymldp}.

As in Section \ref{sec:posasym}, we again switch to the asymptotic regime, so that as $n$ is getting large, $k=\beta n$, $m=\alpha n$, $\kmu=(1-\mu)(n-k)$. where $\beta$, $\alpha$, and $\mu$ are fixed constants independent of $n$. Also, we again set $\rho\triangleq \lim_{n\rightarrow\infty}\frac{l}{n}$, $\rho_1\triangleq \lim_{n\rightarrow\infty}\frac{l_1}{n}$, and $\rho_2\triangleq \lim_{n\rightarrow\infty}\frac{l_2}{n}$. Now, let $S^{(box)}_{\rho_1}=S^{(box)}_{\rho_2}=(\max(0,\beta+(1-\mu)(1-\beta)-\rho),\min((1-\mu)(1-\beta),1-\rho))$. Then, as $n\rightarrow\infty$, from (\ref{eq:finalthm1}) we have
\begin{equation}
\lim_{n\rightarrow \infty}\frac{\log(p^{(box)}_{err}(k,m,n,\kmu))}{n}  =  \max\{\max_{\rho\geq \alpha,\rho_1\in S^{(box)}_{\rho_1}}\lim_{n\rightarrow \infty} \frac{\log(\zeta^{(\infty,box)}_1)}{n},\max_{\rho\geq \alpha,\rho_2\in S^{(box)}_{\rho_2}}\lim_{n\rightarrow \infty}\frac{\log(\zeta^{(\infty,box)}_2)}{n}\},
\label{eq:asym1}
\end{equation}
where
\begin{eqnarray} \label{eq:asym2}
  \lim_{n\rightarrow \infty} \frac{\log(\zeta^{(\infty,box)}_1)}{n} &=&
  \lim_{n\rightarrow \infty} \lp\frac{\log(c^{(l,l_1,box)}_{1})}{n}+
  \frac{\log(\phiint(0,F^{(l,l_1,box)}_1))}{n}+\frac{\log(\phiext(F^{(l,l_1,box)}_1,C^{(box)}_w))}{n}\rp \nonumber \\
  \lim_{n\rightarrow \infty} \frac{\log(\zeta^{(\infty,box)}_2)}{n} &=&
  \lim_{n\rightarrow \infty} \lp\frac{\log(c^{(l,l_2,box)}_{2})}{n}+
  \frac{\log(\phiint(0,F^{(l,l_2,box)}_2))}{n}+\frac{\log(\phiext(F^{(l,l_2,box)}_2,C^{(box)}_w))}{n}\rp.\nonumber \\
\end{eqnarray}
From (\ref{eq:intanal4}) and (\ref{eq:intanal5}) we have
\begin{eqnarray}
\lim_{n\rightarrow \infty} \frac{\log(c^{(l,l_1,box)}_{1})}{n} & = & \lim_{n\rightarrow \infty}\frac{\log\binom{\kmu}{l_1}\binom{n-k-\kmu}{n-l-1-l_1}}{n}\nonumber \\
& = & -(1-\mu)(1-\beta) H\lp\frac{\rho_1}{(1-\mu)(1-\beta)}\rp
-\mu(1-\beta)H\lp\frac{1-\rho-\rho_1}{\mu(1-\beta)}\rp,\label{eq:asym3a}
\end{eqnarray}
and
\begin{eqnarray}
\lim_{n\rightarrow \infty} \frac{\log(c^{(l,l_2,box)}_{2})}{n} & = & \lim_{n\rightarrow \infty}\frac{\log\binom{\kmu}{l_2}\binom{n-k-\kmu}{n-l-l_2}}{n}\nonumber \\
& = & -(1-\mu)(1-\beta) H\lp\frac{\rho_1}{(1-\mu)(1-\beta)}\rp
-\mu(1-\beta)H\lp\frac{1-\rho-\rho_1}{\mu(1-\beta)}\rp,\label{eq:asym3b}
\end{eqnarray}
where as earlier
\begin{equation}\label{eq:asym4}
  H(x)=x\log(x)+(1-x)\log(1-x).
\end{equation}
Focusing on $l_2=l_1$ and following Section \ref{sec:posasym} (and ultimately \cite{Stojnicl1RegPosfinn}), from (\ref{eq:int1anal8}) and (\ref{eq:int2anal8}) we have
\begin{eqnarray}
\lim_{n\rightarrow \infty} \frac{\log(\phiint(0,F^{(l,l_1,box)}_1)}{n} & \leq & \lim_{n\rightarrow \infty} \frac{\log(\phiint(0,F^{(l,l_2,box)}_2)}{n}.\label{eq:asym4a}
\end{eqnarray}
Choosing $B_F^{(l,l_1,box)}=B_F^{(l,l_1,bin)}$ we have
\begin{equation}\label{eq:asym4ca}
\phiint(0,F^{(l,l_2,box)}_2)=\phiint(0,B_F^{(l,l_1,box)} F^{(l,l_2,box)}_2).
\end{equation}
and, as stated in Section \ref{sec:posasym}, following line by line the arguments from \cite{Stojnicl1RegPosfinn} we conclude
\begin{equation}\label{eq:asym4h}
\phiint(0,F^{(l,l_2,box)}_2)=\phiint(0,F^{(l,l_1,box)}_2)=\phiint(0,B_F^{(l,l_1,box)} F^{(l,l_1,box)}_2) \leq \phiint(0,F^{(l,l_1,box)}_1).
\end{equation}
From (\ref{eq:asym4a}) and (\ref{eq:asym4h}) we finally obtain
\begin{eqnarray}
\lim_{n\rightarrow \infty} \frac{\log(\phiint(0,F^{(l,l_1,box)}_1)}{n}  =  \lim_{n\rightarrow \infty} \frac{\log(\phiint(0,F^{(l,l_2,box)}_2)}{n}.
\label{eq:asym4i}
\end{eqnarray}
From (\ref{eq:int2anal8}) we also have
\begin{eqnarray}
\phiint(0,F^{(l,l_1,box)}_2) &  = &   \frac{2^{l-k}}{2^{l-k}(2\pi)^{\frac{l}{2}}}\int_{-\1_{1\times l}\w_{l_1+1:l+l_1}\geq 0,\w_{l_1+1:\kmu}\leq 0,\w_{\kmu+k+1:l+l_1}\geq 0}e^{-\frac{\w_{l_1+1:l+l_1}^T\w_{l_1+1:l+l_1}}{2}}d\w_{l_1+1:l+l_1}\nonumber \\
&  = & \frac{1}{2^{l-k}}P(-\1_{1\times l}\w_{l_1+1:l+l_1}\geq 0),\label{eq:asym6}
\end{eqnarray}
where on the right side of the last equality one can think of the elements of $\w_{l_1+1:\kmu}$ as being the i.i.d. negative standard half normals, the elements of $\w_{\kmu+1:\kmu+k}$ as being the i.i.d. standard normals, and the elements of $\w_{k+1:l+l_1}$ as being the i.i.d. standard half normals. The definition of the large deviations principle then gives
\begin{eqnarray}
\lim_{n\rightarrow \infty} \frac{\log(\phiint(0,F^{(l,l_1,box)}_2))}{n} & = &  \lim_{n\rightarrow \infty} \frac{\log(\frac{1}{2^{l-k}}P(-\1_{1\times l}\w_{l_1+1:l+l_1}\geq 0))}{n} \nonumber \\
& =  &  \min_{\mu_y\geq 0} \lim_{n\rightarrow \infty} \frac{\log(\mE e^{-\mu_y\1_{1\times l}\w_{l_+1:l+l_1}})}{n}-(\rho-\beta)\log(2).\label{eq:asym7a}
\end{eqnarray}
We further also have
\begin{eqnarray}
\lim_{n\rightarrow \infty} \frac{\log(\mE e^{-\mu_y\1_{1\times l}\w_{l_+1:l+l_1}})}{n} &  = &
((1-\beta)\mu-(1-\rho-\rho_1)) \log\lp\mE e^{-\mu_y\w_{\kmu+k+1}}\rp \nonumber \\
& & +((1-\mu)(1-\beta)-\rho_1)\log\lp\mE e^{\mu_y\w_{l_1+1}}\rp+\beta\frac{\mu_y^2}{2}\nonumber \\
& = & ((1-\beta)\mu-(1-\rho-\rho_1)) \log\lp\frac{2}{\sqrt{2\pi}}\int_{0}^{\infty} e^{-\frac{\w_{\kmu+k+1}^2}{2}-\mu_y\w_{\kmu+k+1}}d\w_{\kmu+k+1}\rp  \nonumber \\
& & +((1-\mu)(1-\beta)-\rho_1) \log\lp\frac{2}{\sqrt{2\pi}}\int_{0}^{\infty} e^{-\frac{\w_{l_1+1}^2}{2}+\mu_y\w_{l_1+1}}d\w_{l_1+1}\rp)+\beta\frac{\mu_y^2}{2} \nonumber \\
& = & ((1-\beta)\mu-(1-\rho-\rho_1)) \log\lp\erfc\lp\frac{\mu_y}{\sqrt{2}}\rp\rp \nonumber \\
& & +((1-\mu)(1-\beta)-\rho_1)\log\lp\erfc\lp\frac{-
\mu_y}{\sqrt{2}}\rp\rp +\rho\frac{\mu_y^2}{2}\nonumber \\
& = & ((1-\beta)\mu-(1-\rho-\rho_1)) \log(\erfc(\mu_y))\nonumber \\
& & +((1-\mu)(1-\beta)-\rho_1)\log(\erfc(-\mu_y))+\rho\mu_y^2.\label{eq:asym7b}
\end{eqnarray}
A combination of (\ref{eq:posext1anal10}) and (\ref{eq:ext1anal10}) gives
\begin{multline}
\lim_{n\rightarrow \infty} \frac{\log(\phiext(F^{(l,l_1,box)}_1,C^{(box)}_w))}{n} =
\lim_{n\rightarrow \infty} \frac{\log(\phiext(F^{(l,l_1,bin)}_1,C^{(bin)}_w))}{n}\\
 =  \max_{\g_{n-l}\geq 0} \lp -\rho \g_{n-l}^2+(1-\rho-\rho_1)\log\lp \frac{1}{2}\erfc(-\g_{n-l})\rp
 +\rho_1\log\lp \frac{1}{2}\erfc(\g_{n-l})\rp \rp.
\label{eq:asym8}
\end{multline}
From (\ref{eq:ext2anal2}) we have
\begin{eqnarray}
\lim_{n\rightarrow \infty} \frac{\log(\phiext(F^{(l,l_2,box)}_2,C^{(box)}_w))}{n}
 =  \lim_{n\rightarrow \infty} \frac{\log\lp\frac{1}{2^{n-l}}\rp}{n}=-(1-\rho)\log(2).
\label{eq:asym9}
\end{eqnarray}
For $\g_{n-l}=0$ in (\ref{eq:asym8}) we have $\lim_{n\rightarrow \infty} \frac{\log(\phiext(F^{(l,l_1,box)}_1,C^{(box)}_w))}{n}=-(1-\rho)\log(2)$ which implies that  \begin{eqnarray}
\lim_{n\rightarrow \infty} \frac{\log(\phiext(F^{(l,l_2,box)}_2,C^{(box)}_w))}{n}
 \leq  \lim_{n\rightarrow \infty} \frac{\log(\phiext(F^{(l,l_1,box)}_1,C^{(box)}_w))}{n}.
\label{eq:asym10}
\end{eqnarray}
A combination of (\ref{eq:asym1}), (\ref{eq:asym2}), (\ref{eq:asym3a}), (\ref{eq:asym3b}), (\ref{eq:asym4i}), and (\ref{eq:asym10}) gives
\begin{eqnarray}
\lim_{n\rightarrow \infty}\frac{\log(p^{(box)}_{err}(k,m,n,\kmu))}{n} & = & \max\{\max_{\rho\geq \alpha,\rho_1\in S^{(box)}_{\rho_1}}\lim_{n\rightarrow \infty} \frac{\log(\zeta^{(\infty,box)}_1)}{n},\max_{\rho\geq \alpha}\lim_{n\rightarrow \infty}\frac{\log(\zeta^{(\infty,box)}_2)}{n}\}\nonumber \\
& = & \max_{\rho\geq \alpha,\rho_1\in S^{(box)}_{\rho_1}}\lim_{n\rightarrow \infty} \frac{\log(\zeta^{(\infty,box)}_1)}{n}.
\label{eq:asym11}
\end{eqnarray}
Relying on (\ref{eq:asym2}), (\ref{eq:asym3a}), (\ref{eq:asym3b}), (\ref{eq:asym7a}), (\ref{eq:asym7b}), and (\ref{eq:asym8}), one can rewrite (\ref{eq:asym11}) in the following way
\begin{equation}
\lim_{n\rightarrow \infty}\frac{\log(p^{(box)}_{err}(k,m,n,\kmu))}{n}
  =  \max_{\rho\geq \alpha,\rho_1\in S^{(box)}_{\rho_1}}\lim_{n\rightarrow \infty} \frac{\log(\zeta^{(\infty,box)}_1)}{n}
 \max_{\rho\geq \alpha,\rho_1\in S^{(box)}_{\rho_1}} (\psicom+\psiint+\psiext),\label{eq:asym12a}
\end{equation}
where
\begin{eqnarray}
\psicom & = &  -(1-\mu)(1-\beta) H\lp\frac{\rho_1}{(1-\mu)(1-\beta)}\rp
-\mu(1-\beta)H\lp\frac{1-\rho-\rho_1}{\mu(1-\beta)}\rp \nonumber \\
\psiint & = & \min_{\mu_y\geq 0} \lp((1-\beta)\mu-(1-\rho-\rho_1)) \log(\erfc(\mu_y)) +((1-\mu)(1-\beta)-\rho_1)\log(\erfc(-\mu_y))+\rho\mu_y^2\rp\nonumber \\
& & -(\rho-\beta)\log(2)\nonumber \\
\psiext & = & \max_{\g_{n-l}\geq 0} \lp -\rho \g_{n-l}^2+(1-\rho-\rho_1)\log\lp \frac{1}{2}\erfc(-\g_{n-l})\rp
 +\rho_1\log\lp \frac{1}{2}\erfc(\g_{n-l})\rp \rp.\nonumber \\
\label{eq:asym12b}
\end{eqnarray}
Let $\alpha^{(box)}_w$ be the phase-transition value of $\alpha$, i.e. for any $\beta$, let $\alpha^{(box)}_w$ be the critical value of $\alpha$ that ensures $\lim_{n\rightarrow \infty}\frac{\log(p^{(box)}_{err}(k,m,n,\kmu))}{n} =0$ (as \cite{Stojnicl1BnBxasymldp} proves, such an $\alpha$ always exists). Following the reasoning presented in Section \ref{sec:posasym}, the optimal $\rho$ is equal to $\alpha$ and the outer optimization in (\ref{eq:asym12a}) can be removed. This is true for both regimes $\alpha\geq \alpha^{(box)}_w$ and $\alpha\leq \alpha^{(box)}_w$. The only difference is that for  $\alpha\leq \alpha^{(box)}_w$ one looks at the following complementary version of (\ref{eq:anal3})
\begin{equation}
1-2\sum_{l=m-2j-1,j\in \mN_0,l\geq k-1} \sum_{F^{(l,box)}\in \calF^{(l,box)}}\phiint(0,F^{(l,box)})\phiext(F^{(l,box)},C^{(box)}_w)=1-p^{(bin)}_{cor},\label{eq:asym13}
\end{equation}
where $p^{(bin)}_{cor}$ is the probability that the solution of (\ref{eq:l1}) is the box-constrained sparse solution of (\ref{eq:boxl0}) and its decay rate is given by (\ref{eq:asym12a}) and (\ref{eq:asym12b}) with $\rho\leq \alpha$. \cite{Stojnicl1RegPosfinn} uses the same line of reasoning applied for the probability of error to draw the conclusions about the optimality of $\rho$ in the context of the probability of being correct. (\ref{eq:asym12a}) and (\ref{eq:asym12b}) can then be used to numerically determine the PT and LDP curves of the box $\ell_1$. On the other hand, \cite{Stojnicl1BnBxasymldp} goes much further and explicitly solves (\ref{eq:asym12a}) and (\ref{eq:asym12b}).

\section{Conclusion}
\label{sec:conc}

In this paper we considered random linear systems with the so-called binary and box-constrained solutions. We focused on two modification of the standard $\ell_1$ heuristic (which we referred to as the binary and box $\ell_1$). The modifications that we considered are particularly tailored to increase the efficiency of the standard $\ell_1$ when facing additional binary/box structuring. For both, the binary and the box $\ell_1$, we provided a precise finite dimensional performance analysis. The obtained results complement the ones from our companion paper \cite{Stojnicl1BnBxasymldp}, where the same $\ell_1$ modifications were considered in an infinite-dimensional asymptotic regime and their phase-transitions (PT) and large deviations (LDP) were discussed. Given the systematic type of analysis that we presented, various other extensions are possible. They typically assume a few routine adjustments to the main concepts developed here and in a few of our earlier works. For some of the extensions that we view as the most interesting we will in a couple of forthcoming papers present what kind of change in the final results these adjustments produce.

\begin{singlespace}
\bibliographystyle{plain}
\bibliography{l1bnbxfinn1Refs}
\end{singlespace}

\end{document}